\documentclass[11pt]{article}
\usepackage[
  margin=1.5cm,
  includefoot,
  footskip=30pt,
]{geometry}
\usepackage{amsfonts,amsmath,amssymb,amsthm,appendix,bm,xspace,thmtools,thm-restate,asymptote,mathrsfs,array}
\usepackage{multirow,graphicx,algorithm,algorithmic,caption,subcaption,url,cite,float}
\usepackage{mathtools}
\usepackage{tikz}
\usetikzlibrary{shapes.geometric}
\usepackage[colorlinks=true,urlcolor=blue,citecolor=blue]{hyperref}
\usepackage{todonotes}
\usepackage{array} 
\usepackage[noabbrev]{cleveref}
\usepackage{listings}
\usepackage{pgfplotstable,booktabs}
\usepackage{xpatch}
\usepackage{thmtools}
\usepackage{apptools}
\usepackage{titlesec}
\usepackage[figuresleft]{rotating}
\makeatletter
\xpatchcmd{\thmt@restatable}{\csname #2\@xa\endcsname\ifx\@nx#1\@nx\else[{#1}]\fi}%
{\IfAppendix{\csname #2\@xa\endcsname}{\csname #2\@xa\endcsname[{#1}]}}
\makeatother

\crefname{pr}{Problem}{Problems}

\newcommand{\bdb}[1]{{#1}}

\newtheorem{theorem}{Theorem}[section]

\newtheorem{lemma}[theorem]{Lemma}
\newtheorem{definition}[theorem]{Definition}
\theoremstyle{remark}

\newcommand{\norm}[1]{\left\|{#1}\right\|}
\newcommand{\brac}[1]{\left( #1 \right) }
\newcommand{\op}[1]{ \operatorname{#1} }
\newcommand{\normTV}[1]{\left\| #1 \right\| _{\op{TV}}}
\newcommand{\abs}[1]{\left| #1 \right|}
\newcommand{\normone}[1]{\left\| #1 \right\| _{1}}
\newcommand{\norminf}[1]{\left\|#1\right\|_{\infty}}
\newcommand{\enva}[1]{\abs{#1}_\downarrow}
\newcommand{\envan}[1]{\norm{#1}_\downarrow}
\newcommand{\env}[1]{#1_\infty}
\newcommand{\LB}{\mathrm{LB}}
\newcommand{\wt}{\widetilde}

\newcommand{\gridsep}{\zeta}
\newcommand{\spikesep}{\Delta}
\newcommand{\sampleprox}{\gridsep}

\newcommand{\bcoeff}{\alpha}
\newcommand{\wocoeff}{\beta}
\newcommand{\wtcoeff}{\gamma}
\newcommand{\bcoeffmin}{\bcoeff_{\LB}}
\newcommand{\bcone}{\kappa^B}
\newcommand{\bctwo}{\mu^B}
\newcommand{\bcthree}{\rho^B}
\newcommand{\wocone}{\kappa^{W^1}}
\newcommand{\woctwo}{\mu^{W^1}}
\newcommand{\wocthree}{\rho^{W^1}}
\newcommand{\wtcone}{\kappa^{W^2}}
\newcommand{\wtctwo}{\mu^{W^2}}
\newcommand{\wtcthree}{\rho^{W^2}}
\newcommand{\coeffone}{\kappa}
\newcommand{\coefftwo}{\mu}
\newcommand{\coeffthree}{\rho}
\newcommand{\signs}{\tau}

\newcommand{\II}{I}
\newcommand{\CC}{\mathcal{C}}
\newcommand{\BB}{\mathcal{B}}
\newcommand{\WW}{\mathcal{W}}
\newcommand{\WWo}{\mathcal{W}^1}
\newcommand{\WWt}{\mathcal{W}^2}
\newcommand{\DoB}{\mathcal{B}_x}
\newcommand{\DoWo}{\mathcal{W}^1_x}
\newcommand{\DoWt}{\mathcal{W}^2_x}
\newcommand{\DtB}{\mathcal{B}_y}
\newcommand{\DtWo}{\mathcal{W}^1_y}
\newcommand{\DtWt}{\mathcal{W}^2_y}
\newcommand{\DtWti}{(\DtWt)^{-1}}
\renewcommand{\SS}{\mathcal{S}}
\newcommand{\SSo}{\SS_1}
\newcommand{\SSt}{\SS_2}
\newcommand{\SSth}{\SS_3}
\newcommand{\SSthi}{\SSth^{-1}}
\newcommand{\SSoi}{\SSo^{-1}}
\newcommand{\BBp}{\BB^{(1)}}
\newcommand{\WWp}{\WW^{(1)}}
\newcommand{\WWpi}{(\WWp)^{-1}}
\newcommand{\MM}{M}

\newcommand{\KK}{\mathcal{K}}

\newcommand{\tQ}{\tilde{Q}}
\newcommand{\tq}{\tilde{q}}
\newcommand{\RR}{\mathbb{R}}

\newcommand{\epsbump}{\varepsilon_{\mathrm{B}}}
\newcommand{\epswave}{\varepsilon_{\mathrm{W}}}

\newcommand{\UU}{\mathcal{U}}

\newcommand{\envEV}[1]{\enva{\lambda(#1)}}

\newcommand{\ltmaxwo}{\envEV{W^1}}
\newcommand{\ltmaxwt}{\envEV{W^2}}
\newcommand{\ltamaxb}{\envEV{B}}

\newcommand{\Dt}{D(B)}

\newcommand{\sign}{\operatorname*{sign}}

\newcommand{\sdfraction}{0.65}

\numberwithin{equation}{section}

\pdfoutput = 1
\allowdisplaybreaks
\title{A Sampling Theorem for Deconvolution in Two Dimensions}

\author{Joseph McDonald\thanks{Joint first authors} \thanks{Courant Institute of Mathematical Sciences, New York University}\hspace{0.4cm}
Brett Bernstein\footnotemark[1] \footnotemark[2]
\hspace{0.4cm} Carlos Fernandez-Granda\footnotemark[2] \thanks{Center for Data Science,
    New York University}}

\date{}

\begin{document}

\maketitle

\vspace{-0.3in}

\begin{abstract}
\noindent
This work studies the problem of estimating a two-dimensional superposition of point sources or spikes from samples of their convolution with a Gaussian kernel. Our results show that minimizing a continuous counterpart of the $\ell_1$ norm exactly recovers the true spikes if they are sufficiently separated, and the samples are sufficiently dense. In addition, we provide numerical evidence that our results extend to non-Gaussian kernels relevant to microscopy and telescopy. 
\end{abstract}


\section{Introduction}

Deconvolution is an inverse problem where the goal is to estimate a signal $\mu$ from measurements $y$ modeled as the convolution of $\mu$ with a kernel $K$. More specifically, the measurements $y$ represent samples of the convolved signal observed at certain points $s_i\in\RR^d$,
\begin{equation}
y_i = (K \ast \mu)(s_i),\quad i=1,\ldots,n.
\end{equation}
This problem has applications in various fields including ultrasound \cite{jensen,yu}, optics \cite{broxton}, microscopy \cite{mcnally}, and geology \cite{claerbout,taylor1979deconvolution,chapman}. Often the signal is well modeled as a super-position of point sources, such as celestial objects in astronomical images \cite{lucy,richardson}, fluorescent probes in microscopy \cite{zhu2012}, or neural spike trains in neuroscience \cite{neural-ekanadham}. In these settings the convolution kernel $K$ represents the impulse response of a particular system, e.g. the point spread function of an optical lens. 

Mathematically, a signal $\mu$ consisting of point sources can be represented as an atomic signed measure on $\RR^d$:
\begin{equation}
    \mu := \sum_{t_j\in T} a_j\delta_{t_j},
    \label{eq:signal}
\end{equation}
where $\delta_{t_j}$ is a Dirac measure located at $t_j$, $T:=\{t_1,\ldots,t_N\}\subset\RR^d$ is the support of the signal, and $a_1,\ldots,a_N\in\RR$ are the amplitudes.  The samples are given by
\begin{equation}
    y_i = (K \ast \mu)(s_i) = 
    \sum_{t_j\in T} a_j K(s_i-t_j),\quad
    i = 1,2,\ldots,n.
    \label{eq:samples}
\end{equation}
A two-dimensional example ($d=2$) of a signal and the associated samples is shown in \Cref{fig:spikedata}. 

\begin{figure}[!t]
\begin{subfigure}[h]{0.3\textwidth}
\includegraphics[scale=\sdfraction]{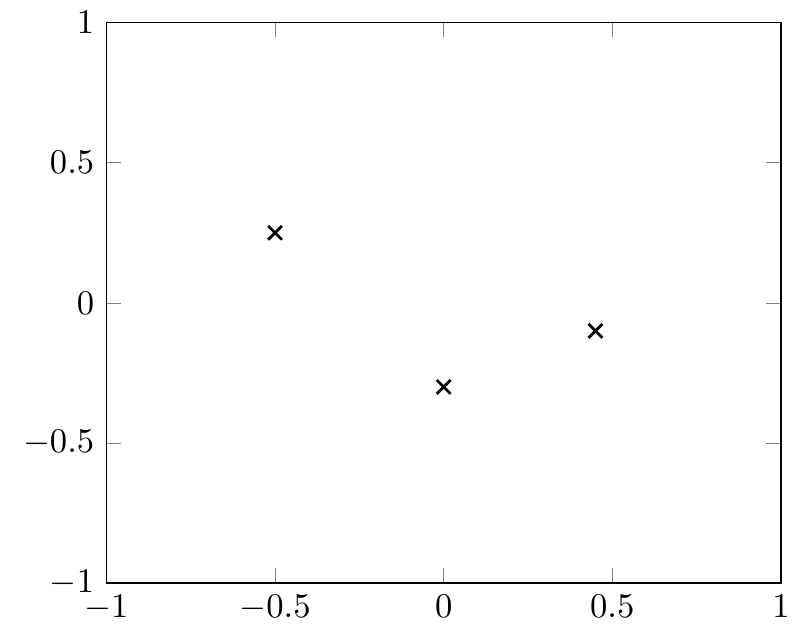}
\caption{\label{fig:spikedata1}}
\end{subfigure}
\begin{subfigure}[h]{0.3\textwidth}
\includegraphics[scale=\sdfraction]{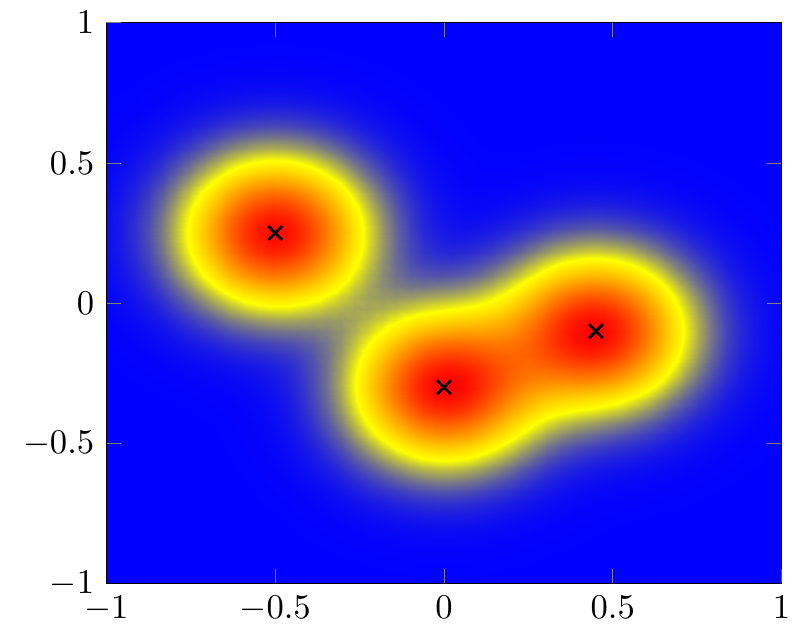}
\caption{\label{fig:spikedata2}}
\end{subfigure}
\begin{subfigure}[h]{0.3\textwidth}
\includegraphics[scale=\sdfraction]{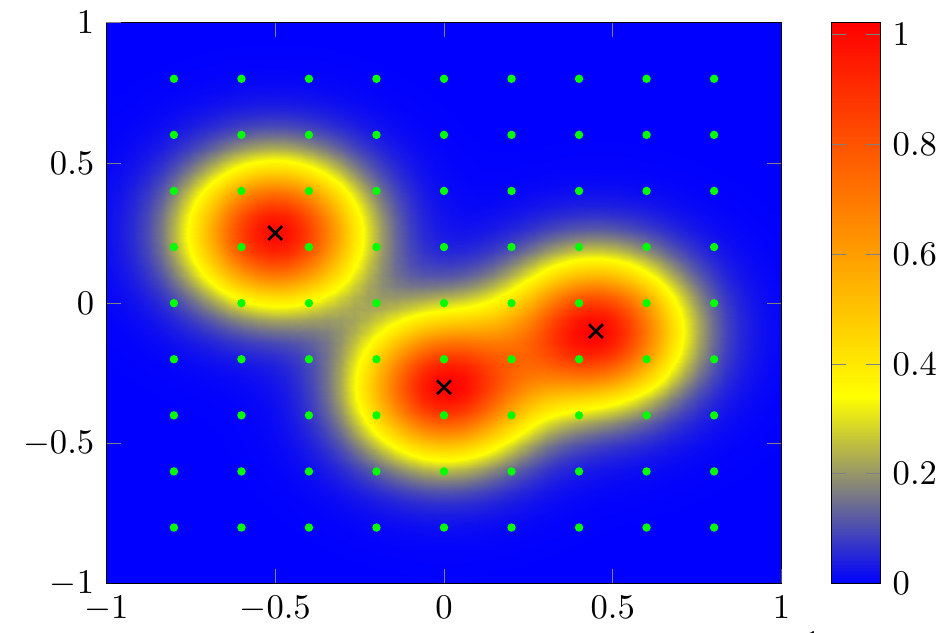}
\caption{\label{fig:spikedata3}}
\end{subfigure}
\caption[Deconvolution example]{Illustration of the deconvolution problem in two dimensions. \Cref{fig:spikedata1} shows a signal consisting of point sources represented by crosses. \Cref{fig:spikedata2} shows the convolution of the signal with a Gaussian kernel. \Cref{fig:spikedata3} shows a grid of sample locations in green.
}
\label{fig:spikedata}
\end{figure}

In the 1970s and 1980s, geophysicists working on reflection seismology developed numerical methods for the deconvolution problem based on $\ell_1$-regularized least squares \cite{claerbout,taylor1979deconvolution,chapman}.
The method works well in practice and has been applied in numerous settings such as marine seismic data \cite{chapman}, signal processing \cite{mallat1999wavelet}, and ultrasound imaging \cite{bendory64stable}.

In \cite{bernstein2018}, the authors developed a theoretical framework for analyzing deconvolution via $\ell_1$-norm minimization in one-dimension ($d=1$).  To allow for arbitrary support discretizations, they considered a continuous analog of the $\ell_1$-norm known as the total-variation (TV) norm \cite[Section~3.1]{folland}:
\begin{equation}
\begin{aligned}
\underset{\tilde{\mu}}{\op{minimize}} \quad& \normTV{ \tilde{\mu} } \\ \text{subject to}
\quad& (K*\tilde{\mu})(s_i) = y_i, \quad i=1,\ldots,n,
\end{aligned}
\label{pr:TVnormIntro}
\end{equation}
where $\tilde{\mu}$ is minimized over the space of signed measures.  Their main result \cite[Theorem 2.4]{bernstein2018} characterizes when the solution to problem~\eqref{pr:TVnormIntro} exactly recovers the signal $\mu$ for Gaussian kernels, and Ricker wavelets (a popular model for impulse responses in geophysics). Stated simply, the result shows that exact recovery is possible when the support is sufficiently separated, and each support element has two nearby samples.  The authors also show that $\ell_1$-norm minimization robustly solves the deconvolution problem when the samples are corrupted by different types of noise.
 
In this paper we extend the results of \cite{bernstein2018} to two dimensions.  This is significant because many applications of deconvolution, specifically imaging applications such as microscopy or telescopy involve two-dimensional data.
Our theory establishes that $\ell_1$-norm minimization achieves exact deconvolution of 2D point sources as long as there are three samples per source, and the sources are separated by a certain minimum distance, as in 1D. The proof relies on a dual-certificate construction, which can be used to derive robustness guarantees. \bdb{In contrast to the 1D case, proving the validity of the certificate in 2D for any possible configuration of sources with bounded minimum separation requires a careful geometric analysis, which is our main technical contribution.} The paper is structured as follows. \Cref{sec:MainResults} presents our main theoretical result, a theorem establishing exact recovery via convex programming for Gaussian deconvolution problems in two dimensions. \Cref{sec:Proof} describes the certificate construction used to prove the main result. Finally, in \Cref{sec:Numerical} we provide numerical experiments illustrating the performance of the method for the Gaussian kernel, and also for other two-dimensional kernels relevant to microscopy and telescopy.

\section{Main Results}\label{sec:MainResults}
Our main result is a sampling theorem for deconvolution via convex optimization in two dimensions.  We show that solving problem~\eqref{pr:TVnormIntro} achieves exact recovery under certain conditions on the spike and sample locations.  For concreteness and brevity, we fix $K$ to be the Gaussian kernel given by
\begin{equation}
    K(t) := \exp\brac{-\frac{\|t\|^2}{2\sigma^2}},
\end{equation}
where $\|t\|$ denotes the standard Euclidean norm on $\RR^2$.  Our results extend to other Gaussian-like kernels, with evidence given in \Cref{sec:Numerical}.

We assume the sample locations $s_1,\ldots,s_n\in\RR^2$ form a uniformly-spaced two-dimensional grid that surrounds the spike locations $t_j\in T$.  Uniform sampling is a natural choice when no prior assumptions are made on the spike locations.  \Cref{fig:spikedata} depicts an instance of the two dimensional Gaussian deconvolution problem with uniform samples.

Our main result, \Cref{thm:Exact}, shows that we can exactly recover $\mu$ when the spike locations are sufficiently separated and the sampling grid is dense enough.  In the next section we motivate these conditions.

\subsection{Minimum Separation and Grid Spacing}\label{sec:MinimumSeparationandGridSpacing}

Without assumptions on the underlying signals, deconvolution is an ill-posed problem. The numerical experiments in Section~\ref{sec:SVD} show that when two signals have very clustered supports, their difference may lie almost in the nullspace of the convolution operator (see also Section~2.1.1 in~\cite{bernstein2018} and Section 3.2 in \cite{superres} for more details). Following previous works on deconvolution~\cite{bernstein2018} and point-source super-resolution~\cite{superres}, we restrict our attention to signals with supports satisfying a \textit{minimum-separation} condition.
\begin{definition}\label{def:minsep}
The minimum separation of the support $T=\{t_1,\ldots,t_N\}\subset\RR^2$ of a signal is
\begin{equation}
\spikesep(T):=\min_{i\neq i'}\norm{t_i-t_{i'}}.
\end{equation}
\end{definition}
In \Cref{sec:SVD} we present numerical evidence that a minimum separation of at least $\sigma$ is needed for the deconvolution problem to be well-posed, in the sense that 
there exist pairs of signals with smaller minimum separation whose difference approximately lies in the nullspace of the 2D convolution operator. \bdb{Section~2.1.1 in~\cite{bernstein2018} provides an explicit example of two nonzero signals with small minimum separation that produce almost indistinguishable samples.}

Our ability to robustly solve the deconvolution problem depends on the relative location of the samples and the support of the true signal. For convolution kernels with decaying tails, like the Gaussian, samples that are too distant from the support contain almost no information. If the samples lie on an uniform grid, robust recovery is only possible if the grid spacing is small enough to ensure that there are always some measured samples close to the spike locations (see Section~2.1.2 in \cite{bernstein2018} for a more detailed discussion).

\begin{definition}
\label{def:gridsep}
The grid spacing $\gridsep>0$ of the sampling grid is the distance between consecutive sample points in both directions.  The samples on the grid can be expressed as
\begin{equation}
    s_1 + \gridsep\cdot(p,q),
\end{equation}
where $s_1\in\RR^2$, $p=0,\ldots,S_1-1$, $q=0,\ldots,S_2-1$, and $S_1,S_2$ are the grid dimensions.
\end{definition}
The minimum separation $\spikesep(T)$ and grid spacing $\gridsep$ are depicted in \Cref{fig:minsepgridsep}.

\begin{figure}[t!]
\centering
\includegraphics{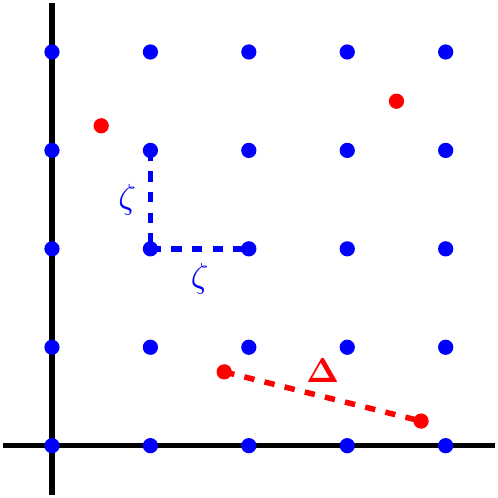}
\caption[Minimum Separation and Grid Spacing]{The minimum separation $\spikesep$ is the smallest distance between any pair of spikes (red points), while the grid spacing $\gridsep$ is the distance between the sample points (depicted in blue) in either direction.}
\label{fig:minsepgridsep}
\end{figure}

\subsection{Sampling Theorem for Exact Recovery}\label{sec:SamplingTheorem}

The main contribution of this paper is a sampling theorem that establishes $\ell_1$-norm minimization as an accurate method for deconvolution in two dimensions.
We make our analysis independent of the discretization of the signal support by considering the total variation (TV) norm for sparse measures, defined as
\begin{equation}
\normTV{\mu} = \sup_{f:|f|\leq 1}\int f\,d\mu,
\end{equation}
where the supremum is over continuous functions bounded by one in absolute value \cite{folland}.
This norm is analogous to the $\ell_1$ norm in discrete spaces. In fact, for atomic measures $\sum_{j=1}^N a_j\delta_{t_j}$ the TV norm is exactly the $\ell_1$ norm of the vector of coefficients, $\sum\abs{a_j}$.

We show that TV-norm minimization achieves exact deconvolution under certain conditions on the grid spacing and minimum separation.
Roughly speaking, if the grid spacing $\gridsep$ is slightly less than the standard deviation $\sigma$ of the convolution kernel $K$ and the support $T$ of the signal $\mu$ has a large enough minimum separation then $\mu$ is the unique solution of the convex program
\begin{equation}
\begin{aligned}
\underset{\tilde{\mu}}{\op{minimize}} \quad& \normTV{ \tilde{\mu} } \\ \text{subject to}
\quad& (K*\tilde{\mu})(s_i) = y_i, \quad i=1,\ldots,n.
\end{aligned}
\label{pr:TVnorm}
\end{equation}
\begin{theorem}[Proof in \Cref{sec:Proof}]\label{thm:Exact}
Let $\mu$ be a signal defined by \eqref{eq:signal}. The corresponding data are of the form \eqref{eq:samples}, where $K$ is the Gaussian kernel. Assume that the signal support has a minimum separation $\spikesep(T)$ and samples are measured on a square grid with spacing $\gridsep$, where the support of $\mu$ lies within the perimeter of the grid edges \bdb{and $\spikesep(T)$ and $\gridsep$ are expressed in units of $\sigma$}. If the pair $(\spikesep(T),\gridsep)$ lies in the orange region in \Cref{fig:RecoveryMapFirst},
then $\mu$ is the unique solution to problem \eqref{pr:TVnorm}.
\end{theorem}
\begin{figure}[t!]
\begin{subfigure}{0.5\textwidth}
\centering
\includegraphics[scale=0.5]{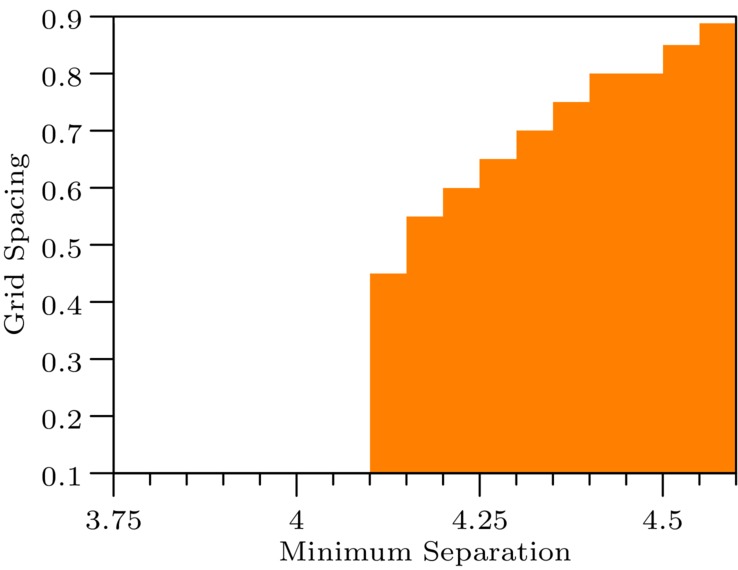}
\caption{}
\label{fig:RecoveryMapFirst}
\end{subfigure}
\begin{subfigure}{0.5\textwidth}
\centering
\includegraphics[scale=0.5]{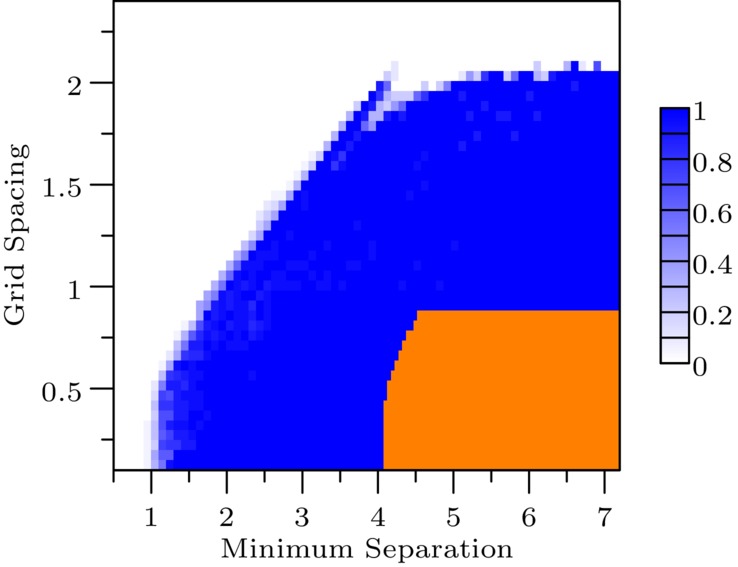}
\caption{}
\label{fig:ExactHeatMapAug}
\end{subfigure}
\caption[Recovery with $\ell_1$ minimization]{\Cref{fig:RecoveryMapFirst} shows the region where our theory guarantees exact recovery via convex programming for a range of minimum separation and grid spacing values (in units of $\sigma$). \Cref{fig:ExactHeatMapAug} shows the theoretical region superimposed on the region where we observe exact recovery numerically (see \Cref{sec:ExactSimulation}).}
\label{fig:RecoveryFigures}
\end{figure}
\Cref{thm:Exact} is a sampling theorem for deconvolution in two dimensions, providing exact recovery guarantees that only depend on the minimum separation and the resolution of the sampling grid.
\bdb{\Cref{tab:RecoveryTable} in \Cref{sec:RecoveryTable} provides a quantitative description of the boundaries of the recovery region shown in \Cref{fig:RecoveryMapFirst}.}
\Cref{fig:ExactHeatMapAug} compares our theoretical guarantees with the empirical performance of the method (see \Cref{sec:Numerical} for a description of the numerical experiments).
We prove \Cref{thm:Exact} by establishing the existence of a dual-feasible vector, known as a \emph{dual certificate} in the literature, as described in \Cref{sec:Proof}. 

An immediate corollary of \Cref{thm:Exact} is a recovery guarantee for $\ell_1$-norm minimization in a discretized setting, where the signal lies on a predefined grid.

\begin{restatable}[Proof in \Cref{sec:DualCertProof}]{corollary}{DiscreteTV}\label{cor:DiscreteTV}
Assume that the support $T$ of the measure $\mu$ in equation \eqref{eq:signal} lies on a known discretized grid $G\subset\RR^2$, and that the data are of the form \eqref{eq:samples}, where $K$ is the Gaussian kernel. Then if the minimum separation $\spikesep(T)$ and grid spacing $\gridsep$ satisfy the conditions of \Cref{thm:Exact}, the coefficients $a_1,\ldots,a_{\abs{G}}$ are the unique solution to \begin{equation}
\begin{aligned}
\underset{\tilde{a}\in\RR^{|G|}}{\op{minimize}} \quad& \normone{ \tilde{a} } 
\\ \text{subject to}
\quad& \sum_{t_j\in \bdb{G}} \tilde{a}_jK(s_i-t_j) = y_i, \quad i=1,\ldots,n.
\end{aligned}
\label{pr:TVnormDisc}
\end{equation}
\end{restatable}

\subsection{Noisy Measurements and Discretization Errors}
In practice, real measurements are corrupted by noise.  We can account for noisy measurements
by perturbing equation~\eqref{eq:samples} with an additive noise vector $z\in\RR^n$:
\begin{equation}
    \hat{y}_i = z_i + \sum_{t_j\in T}a_j K(s_i-t_j),\quad i=1,\ldots,n.
\end{equation}
To adapt problem~\eqref{pr:TVnorm} to noisy measurements, we relax the 
data consistency constraint from an equality to an inequality:
\begin{equation}
\begin{aligned}
\underset{\tilde{\mu}}{\op{minimize}} \quad& \normTV{ \tilde{\mu} } \\ \text{subject to}
\quad& \sum_{i=1}^n( (K*\tilde{\mu})(s_i) - \hat{y}_i)^2\leq\xi^2,
\end{aligned}
\label{pr:TVnormNoise}
\end{equation}
where $\xi>0$ is a parameter that must be tuned to the level of noise.  Combining the arguments in \cite{bernstein2018,support_detection} with our dual-certificate construction in \Cref{sec:Proof} yields robustness guarantees for recovering $\mu$ in high signal-to-noise settings. We omit the details for brevity.

In contrast with \Cref{cor:DiscreteTV}, the true support of $\mu$ may not lie on a known discretized grid.  The same techniques used to derive robustness guarantees for additive noise can also give some control over the discretization error.  Sharpening these guarantees is an interesting direction for future research.

\subsection{Related Work}
As mentioned in the introduction, to the best of our knowledge $\ell_1$-norm minimization for deconvolution was originally proposed in the 1970s by researchers from geophysics~\cite{taylor1979deconvolution,claerbout,levy,santosa,debeye1990lp}.
The first theoretical results analyzed random convolution kernels~\cite{hauptToeplitz,romberg2009compressive} using techniques from compressed sensing~\cite{Candes:2005cs,donoho2006compressed}.
\bdb{\cite{decastro} introduced total-variation minimization as a method for recovering sparse signed measures from their generalized moments.}
The dual certificates used in compressed sensing are not directly applicable to our deterministic kernels because the corresponding linear operators do not satisfy incoherence conditions (see \cite{separable} for a more detailed explanation).
The style of proof employed here first appeared in~\cite{superres,superres_new} 
to establish exact recovery guarantees for super-resolution problems satisfying a minimum separation condition.  Subsequent papers build on these results to study noise~\cite{tang2014robust,support_detection,superres_noisy,peyreduval}, missing data~\cite{tang2012offgrid}, and outliers~\cite{fernandez2017demixing}.  

\bdb{Deconvolution via convex programming has been studied in one dimension in \cite{bernstein2018} for arbitrary sampling patterns and in \cite{schiebinger,eftekhari2018sparse,eftekhari2018stable} for nonnegative signals.} Previous works have analyzed two-dimensional deconvolution for randomized measurements~\cite{poon2018dual}, limiting cases of spike arrangements~\cite{poon}, in settings without discrete samples~\cite{bendory64stable,bendory2016robust},
and as part of the larger class of separable nonlinear problems \cite{separable}.
Our proof most closely follows the techniques in \cite{bernstein2018} for the one-dimensional setting.  The key difference \bdb{is the difficulty in characterizing how spikes may cluster in signals with a fixed minimum separation, which requires a careful geometric analysis described in \Cref{sec:Geometry,sec:Norms,sec:qboundproof}. This complication does not arise in one dimension. An alternative proof strategy for the 2D deconvolution problem would build upon the techniques in \cite{separable}, which provides qualitative guarantees through a certificate based on the correlation structure of the measurement operator. This makes it possible to obtain results for a wider variety of measurement operators, sacrificing precise constants. In contrast, in this work our goal is to derive sharper guarantees, which requires a tailored certificate construction.}

\bdb{The finite-rate-of-innovation (FRI)
framework~\cite{vetterli2002sampling,dragotti2007sampling,uriguen2013fri}, an approach to signal recovery inspired by Prony's method~\cite{deProny:tg}, provides an alternative framework to tackle deconvolution problems. In one dimension \cite{vetterli2002sampling} showed that this technique achieves exact deconvolution of point sources without a minimum-separation condition and without discretizing the parameter space, but does not provide robustness guarantees. As explained in Section~\ref{sec:MinimumSeparationandGridSpacing}, such guarantees would require conditions on the signal support.}
\bdb{These results have been extended to multi-dimensional settings in \cite{maravic,shukla}. The methodology is based on annihilating filters designed to recover 2D signals such as superpositions of points sources, lines, and polygons. 
More recently, \cite{pan2016towards} introduced a method to reconstruct FRI signals using nonuniform sampling patterns, 
and applied it to radio interferometry problems. 
}


\section{Proof of \Cref{thm:Exact}}\label{sec:Proof}
In the proof of \Cref{thm:Exact} we use a standardized Gaussian kernel with $\sigma=1$:
\begin{equation}
K(t):=\exp\brac{-\frac{\norm{t}^2}{2}}
\end{equation}
without loss of generality. This is equivalent to expressing $t$ in units of $\sigma$. Some parts of the proof require computations implemented using Mathematica code, which is available at \url{https://github.com/jpmcd/Deconvolution2D}.

\subsection{Dual Certificate}\label{sec:DualCertificate}
We prove \Cref{thm:Exact} by establishing the existence of a function that guarantees exact recovery:
\begin{restatable}[Proof in \Cref{sec:DualCertProof}]{proposition}{DualCert}
\label{prop:DualCert}
Let $T \subseteq \RR^2$ be the support of a signal $\mu$ of the
form~\eqref{eq:signal} and let $S=\{s_i\}$ be the set of sample points from a
sampling grid on $\RR^2$.
If for any sign pattern $\signs\in\{-1,1\}^{|T|}$
there exists a function of the form
\begin{align}
\tQ(t) := \sum_{i=1}^n \tq_i K(s_i-t) \label{eq:dualcomb_K}
\end{align}
with $\tq\in\RR^n$ satisfying
\begin{alignat}{2}
 & \tQ(t_j) = \signs_j, \qquad && \forall t_j \in
  T, \label{eq:conditionQT}\\ 
  & \abs{\tQ(t)} < 1, && \forall t \in
  T^c, \label{eq:conditionQTc}
\end{alignat}
then the unique solution to problem~\eqref{pr:TVnorm} is $\mu$.
\end{restatable}
The proposition establishes that exact recovery is guaranteed by the existence of an \textit{interpolation function} $\tQ$ that interpolates the sign pattern at the signal's support $T$ using scaled copies of the convolution kernel centered at the sample points in $S$. 
An interpolation function for an example with three spikes is depicted in \Cref{fig:ContourPlot}.
The vector $\tq$ is known as a {\em dual certificate} since it is a feasible solution for the dual of problem~\eqref{pr:TVnorm}:
\begin{equation}
\begin{aligned}
\underset{\nu}{\op{maximize}}\quad&\nu^Ty\\
\text{subject to}\quad&\sup_t\abs{\sum_{i=1}^n \nu_iK(s_i-t)}\leq 1.
\end{aligned}\label{pr:DualCert}
\end{equation}
Dual certificates have been widely used to derive recovery guarantees for convex-programming approaches in compressed sensing \cite{candes2006robust}, matrix completion~\cite{candes2012} and phase retrieval \cite{candes2015phase}.

\begin{figure}[t!]
\centering
\includegraphics[scale=0.75]{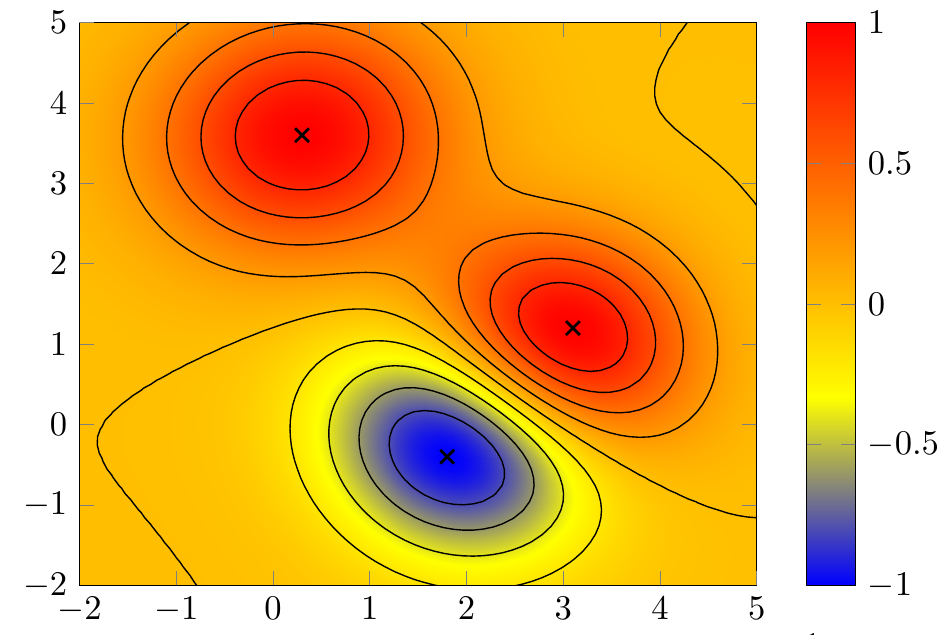}
\caption[Example of a Dual Certificate]{Contour plot of an example of the interpolating function defined in \Cref{prop:DualCert}. The function is a linear combination of copies of the convolution kernel centered \bdb{at samples on a grid with $\gridsep=0.5$}. It interpolates the sign pattern of signal consisting of three spikes and its magnitude is bounded by one.}
\label{fig:ContourPlot}
\end{figure}

By \eqref{eq:conditionQT}, any function $\tQ$ satisfying the conditions of \Cref{prop:DualCert} must interpolate the sign pattern $\signs$ on $T$.
To satisfy condition \eqref{eq:conditionQTc}, $\tQ$ must also have a local extremum at each element of the support:
\begin{equation}\label{eq:conditionQderiv}
\nabla \tQ(t_j)=0\qquad \forall t_j\in T.
\end{equation}
Combining the interpolation and derivative conditions on the support gives a system of $3\abs{T}$ scalar equations, which we refer to as the \textit{interpolation equations}:
\begin{equation}\begin{aligned}\label{eq:Interpolation}
\tQ(t_j) &= \signs_j,\\ 
\nabla \tQ(t_j)&=0,\qquad \forall t_j \in T.
\end{aligned}\end{equation}
The following lemma
establishes the existence of a function satisfying the interpolation equations when the assumptions of \Cref{thm:Exact} are met.
\begin{lemma}[Proof in \Cref{sec:InvertibilityProof}]
\label{lem:Invertibility}
Under the assumptions of \Cref{thm:Exact}, the system of equations \eqref{eq:Interpolation} has a solution.
\end{lemma}
Below we outline the proof of \Cref{lem:Invertibility}, a 
two-dimensional analog of the argument given in \cite{bernstein2018}. To begin, we restrict our focus to the set of $3|T|$ samples obtained by choosing the three samples closest to each spike.  This allows us to express the interpolations equations \eqref{eq:Interpolation} as a $3|T|\times 3|T|$ linear system.
\bdb{Note that at least three samples are needed to determine an individual spike because it is encoded by three parameters: its amplitude and its two-dimensional location.} 

In \Cref{sec:BumpsWaves} we apply a change of basis that approximately diagonalizes this system, and allows us to state conditions guaranteeing invertibility in \Cref{sec:InvertibilityProof}.
In contrast with the one-dimensional case, establishing these invertibility conditions requires a geometric argument that we explain in \Cref{sec:Geometry}.
In \Cref{sec:Norms} we complete the proof of \Cref{lem:Invertibility} by showing that the linear system is invertible under the conditions of \Cref{thm:Exact}.

The proof of \Cref{lem:Invertibility} sketched above yields an explicit interpolation function candidate $Q$ that solves the system \eqref{eq:Interpolation}.  The following lemma shows that this candidate satisfies condition \eqref{eq:conditionQTc}, thereby establishing \Cref{thm:Exact}.
\begin{lemma}[Proof in \Cref{sec:qboundproof}]
\label{lem:QBound}
Under the assumptions of \Cref{thm:Exact}, the interpolation function $Q$
solving system~\eqref{eq:Interpolation} guaranteed by \Cref{lem:Invertibility} satisfies $\abs{ Q(t) } <1$ for all $t\in T^c$.
\end{lemma}


\subsection{Bumps and Waves}
\subsubsection{Interpolation with Two Dimensional Bumps and Waves}\label{sec:BumpsWaves}
To prove \Cref{lem:Invertibility} we construct an explicit interpolating function
\begin{equation}
    Q(t) = \sum_{i=1}^n q_iK(s_i-t),\quad q_i\in\RR,
\end{equation}
that is a solution to the interpolation equations \eqref{eq:Interpolation}. To find a $q\in\RR^n$ that satisfies equation \eqref{eq:Interpolation}, we must solve a $3|T|\times n$ linear system that is hard to analyze directly.  To avoid this difficulty we extend a key technique from \cite{bernstein2018} to the two-dimensional setting: we perform a reparametrization of $Q$ that yields an approximately diagonal $3|T|\times 3|T|$ system.  Formally, we write $Q$ in the form
\begin{equation}\label{eq:combination}
Q(t) = \sum_{j=1}^{|T|} \bcoeff_jB_j(t;s_j^1,s_j^2,s_j^3)+\wocoeff_jW^1_j(t;s_j^1,s_j^2,s_j^3)+\wtcoeff_jW^2_j(t;s_j^1,s_j^2,s_j^3),
\end{equation}
where for each $t_j\in T$, $s_j^1,s_j^2,s_j^3$ denote the three closest samples, and $B_j,W_j^1,W_j^2$ are modified kernels.  Each of these modified kernels is expressed as a linear combination of shifted copies of $K$:
\begin{equation}\label{eq:BWdefs}
\begin{aligned}
B_j(t;s_j^1,s_j^2,s_j^3)&:=\bcone_j K(s_j^1-t) + \bctwo_j K(s_j^2-t)+ \bcthree_j K(s_j^3-t),\\
W^1_j(t;s_j^1,s_j^2,s_j^3)&:=\wocone_j K(s_j^1-t) + \woctwo_j K(s_j^2-t)+ \wocthree_j K(s_j^3-t),\\
W^2_j(t;s_j^1,s_j^2,s_j^3)&:=\wtcone_j K(s_j^1-t) + \wtctwo_j K(s_j^2-t)+ \wtcthree_j K(s_j^3-t).
\end{aligned}
\end{equation}
Below we omit $s_j^i$ from the argument of these functions where convenient for ease of notation.
The $B_j$ functions, which we call \textit{bumps} due to their shape, are defined by the equations:
\begin{equation}\label{eq:Binterp}
\begin{aligned}
B_j(t_j)&=1,&\partial_x B_j(t_j)&=0,&\partial_y B_j(t_j)=0,
\end{aligned}
\end{equation}
for each $t_j\in T$.  Here $\partial_x$ and~$\partial_y$ denote the partial derivatives with respect to the two coordinates of $t_j\in\RR^2$.
Analogously, the \textit{wave} functions $W^1_j$ and $W^2_j$ are defined by
\begin{equation}\label{eq:Winterp}
\begin{aligned}
W^1_j(t_j)&=0,&\partial_x W^1_j(t_j)&=1,&\partial_y W^1_j(t_j)=0,\\
W^2_j(t_j)&=0,&\partial_x W^2_j(t_j)&=0,&\partial_y W^2_j(t_j)=1,
\end{aligned}
\end{equation}
for all $t_j\in T$.
In \Cref{sec:BWExistProof} we provide a detailed proof that the bumps and waves defined above exist
\bdb{
and are uniquely determined by the linear system derived from \eqref{eq:Binterp} and \eqref{eq:Winterp}, with coefficients
\begin{align}
\begin{bmatrix}
\bcone_j&\wocone_j&\wtcone_j\\
\bctwo_j&\woctwo_j&\wtctwo_j\\
\bcthree_j&\wocthree_j&\wtcthree_j
\end{bmatrix}
&=\frac{1}{D}
\begin{bmatrix}
e^{\|s^1_j-t_j\|^2/2}&0&0\\
0&e^{\|s^2_j-t_j\|^2/2}&0\\
0&0&e^{\|s^3_j-t_j\|^2/2}
\end{bmatrix}
\begin{bmatrix}
D_1 & (s^2_{(2)}-s^3_{(2)}) & (s^3_{(1)}-s^2_{(1)})\\
D_2 & (s^3_{(2)}-s^1_{(2)}) & (s^1_{(1)}-s^3_{(1)})\\
D_3 & (s^1_{(2)}-s^2_{(2)}) & (s^2_{(1)}-s^1_{(1)})
\end{bmatrix}
\end{align}
where $t_j=(t_{(1)},t_{(2)})$, $s^i_j=(s^i_{(1)},s^i_{(2})$ and $D\neq0$.
Here
\begin{equation}
\begin{aligned}
D_1&=(s^2_{(1)}-t_{(1)})(s^3_{(2)}-t_{(2)})-(s^3_{(1)}-t_{(1)})(s^2_{(2)}-t_{(2)})\\
D_2&=(s^3_{(1)}-t_{(1)})(s^1_{(2)}-t_{(2)})-(s^1_{(1)}-t_{(1)})(s^3_{(2)}-t_{(2)})\\
D_3&=(s^1_{(1)}-t_{(1)})(s^2_{(2)}-t_{(2)})-(s^2_{(1)}-t_{(1)})(s^1_{(2)}-t_{(2)})
\end{aligned}
\end{equation}
with $D=D_1+D_2+D_3$.
}
We also show that the bump coefficients, $\bcone_j$, $\bctwo_j$, and $\bcthree_j$, are always non-negative.  Examples of bumps and waves are shown in \Cref{fig:demobw}.

\begin{figure}[t!]\centering
\begin{subfigure}{.45\textwidth}
\centering
\includegraphics[scale=0.75]{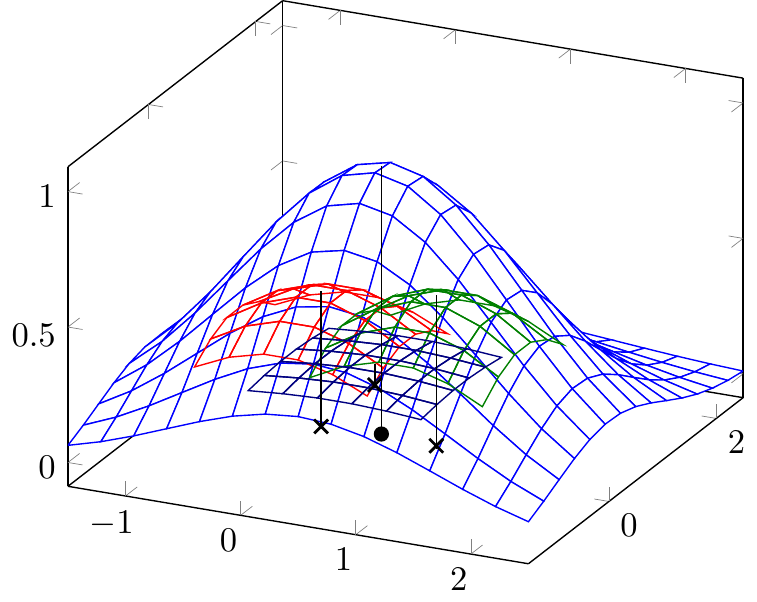}
\end{subfigure}
\begin{subfigure}{.45\textwidth}
\centering
\includegraphics[scale=0.75]{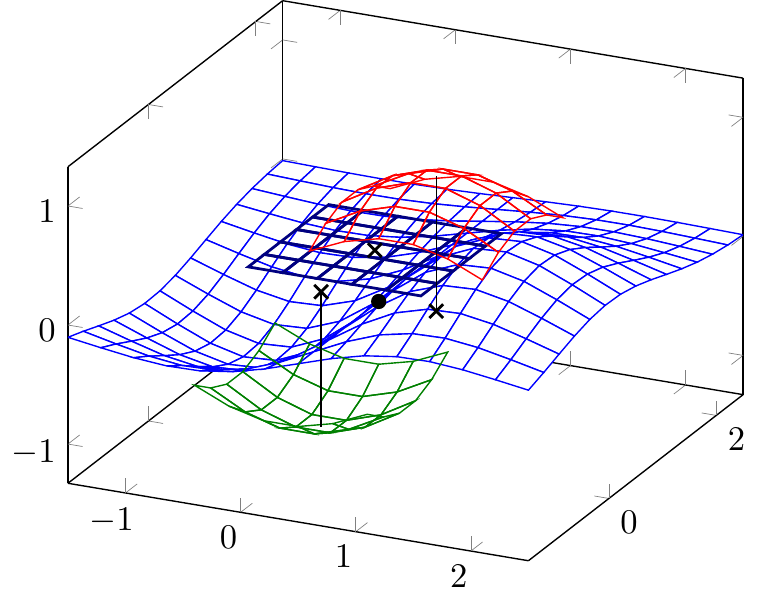}
\end{subfigure}
\caption[Plots of Bump and Wave Functions]{\label{fig:demobw}Examples of bump (left) and wave (right) functions. These functions are plotted in blue, \bdb{and their additive Gaussian components are plotted in red, green and black. The dot and three crosses represent the locations of the spike and the three closest samples respectively. The components in black have very small coefficients making them appear flat.}}
\end{figure}

The intuition behind this reparametrization is that each bump $B_j$ nearly interpolates the sign pattern $\signs_j$ at $t_j$ while the waves $W^1_j$ and $W^2_j$ alter the gradient of $Q$ to correct for the interactions from other bumps and waves.
In terms of the coefficients, this means that
$\bcoeff_j\approx\signs_j$
while both $\wocoeff_j, \wtcoeff_j\approx 0$.

Using the reparametrized $Q$ given in \eqref{eq:combination}, we can express the interpolation equations 
\begin{equation}
Q(t_j) = \signs_j,\quad \nabla Q(t_j)=0,\quad t_j\in T,
\end{equation}
in the matrix form
\begin{equation}
\begin{bmatrix}
\BB & \WWo & \WWt\\
\DoB & \DoWo& \DoWt\\
\DtB & \DtWo& \DtWt
\end{bmatrix}
\begin{bmatrix}
\bcoeff\\
\wocoeff\\
\wtcoeff
\end{bmatrix}
=
\begin{bmatrix}
\signs\\
0\\
0
\end{bmatrix}
\label{eq:BlockEquation}
\end{equation}
where the nine $|T|\times|T|$ block matrices above are defined by 
\begin{equation}
\begin{aligned}\label{eq:MatrixDefinition}
(\BB)_{jk} =&B_k(t_j) & (\WWo)_{jk}=&W^1_k(t_j) & (\WWt)_{jk}=&W^2_k(t_j)\\
(\DoB)_{jk} =&\partial_xB_k(t_j) & (\DoWo)_{jk}=&\partial_xW^1_k(t_j) & (\DoWt)_{jk}=&\partial_xW^2_k(t_j)\\
(\DtB)_{jk} =&\partial_yB_k(t_j) & (\DtWo)_{jk}=&\partial_yW^1_k(t_j) & (\DtWt)_{jk}=&\partial_yW^2_k(t_j),
\end{aligned}
\end{equation}
for $t_j,t_k\in T$.  By construction, the $3|T|\times 3|T|$ matrix equation \eqref{eq:BlockEquation} will be approximately equal to the identity matrix.  We exploit these properties in \Cref{sec:InvertibilityProof} to prove that equation \eqref{eq:BlockEquation} has a solution $(\bcoeff,\wocoeff,\wtcoeff)$.  This yields an interpolation function $Q$ defined by \eqref{eq:combination}, thereby completing the proof of \Cref{lem:Invertibility}.

\subsubsection{Bounding Bumps and Waves}
\label{sec:BWBounds}
In this section we construct radially symmetric upper bounds for the bumps, waves and their derivatives that are used in the proofs of \Cref{lem:Invertibility,lem:QBound}. 
For a fixed grid spacing distance $\gridsep$ (see \Cref{def:gridsep}), we define the following envelope functions for the bump, the waves, their partial derivatives, and the largest absolute eigenvalues of their Hessians:
\begin{equation}
\begin{aligned}\label{eq:Envelopes}
\enva{B}(r)&=\sup_{\substack{\norm{t-t_1}\geq r\\ s_1^1,s_1^2,s_1^3\text{ nearest }t_1}} \abs{B_1(t;s^1_1,s^2_1,s^3_1)},\\
\enva{\partial_z B}(r)&=\sup_{\substack{\norm{t-t_1}\geq r\\ s_1^1,s_1^2,s_1^3\text{ nearest }t_1}} \abs{\partial_z B_1(t;s^1_1,s^2_1,s^3_1)},\\
\enva{\lambda(B)}(r)&=\sup_{\substack{\norm{t-t_1}\geq r\\ s_1^1,s_1^2,s_1^3\text{ nearest }t_1\\\|v\|=1}} \abs{v^T\nabla^2 B_1(t;s^1_1,s^2_1,s^3_1)v}\\
\enva{W^i}(r)&=\sup_{\substack{\norm{t-t_1}\geq r,\\ s_1^1,s_1^2,s_1^3\text{ nearest }t_1}}\abs{W^i_1(t;s^1_1,s^2_1,s^3_1)},\\
\enva{\partial_z W^i}(r)&=\sup_{\substack{\norm{t-t_1}\geq r\\ s_1^1,s_1^2,s_1^3\text{ nearest }t_1}} \abs{\partial_z W^i_1(t;s^1_1,s^2_1,s^3_1)},\\
\enva{\lambda(W^i)}(r)&=\sup_{\substack{\norm{t-t_1}\geq r\\ s_1^1,s_1^2,s_1^3\text{ nearest }t_1\\\|v\|=1}} \abs{v^T\nabla^2 W^i_1(t;s^1_1,s^2_1,s^3_1)v},\\
\end{aligned}
\end{equation}
where $z\in\{x,y\}$. 
\bdb{The suprema are taken over all possible relative positions of the spike $t_1$ and its nearest three samples when the grid spacing is fixed at $\gridsep$.}
By construction the envelopes are monotonically decreasing as $r$ grows.

Our proofs also require non-monotonic upper bounds on the directional derivatives and Hessian eigenvalues of the bumps:
\begin{equation}
\begin{aligned}\label{eq:RawEnvelopes}
\env{D( B)}(r)&=\sup_{\substack{\norm{t-t_1}=r\\ s_1^1,s_1^2,s_1^3\text{ nearest }t_1}} \nabla B_1(t;s^1_1,s^2_1,s^3_1)\cdot \frac{t-t_1}{\norm{t-t_1}},\\
\env{\lambda(B)}(r)&=\sup_{\substack{\norm{t-t_1}=r\\ s_1^1,s_1^2,s_1^3\text{ nearest }t_1\\\|v\|=1}} v^T\nabla^2 B_1(t;s^1_1,s^2_1,s^3_1)v.
\end{aligned}
\end{equation}

In \Cref{sec:Envelopes} we compute piecewise-constant upper bounds of these envelopes for $r\in[0,10]$ and $\gridsep\in[0.1,0.9]$. 
For $r\geq10$ and $\gridsep\leq 2$, \Cref{lem:BumpBounds,lem:WaveBounds,lem:EVTail} establish that the envelopes are upper bounded by $2\cdot10^{-9}$. Examples of the envelopes $\enva{B}$ and $\enva{W^1}$ are shown in \Cref{fig:Envelopes} for a selection of grid spacings.

\begin{figure}
\begin{subfigure}{0.5\textwidth}
    \centering
    \includegraphics[scale=0.5]{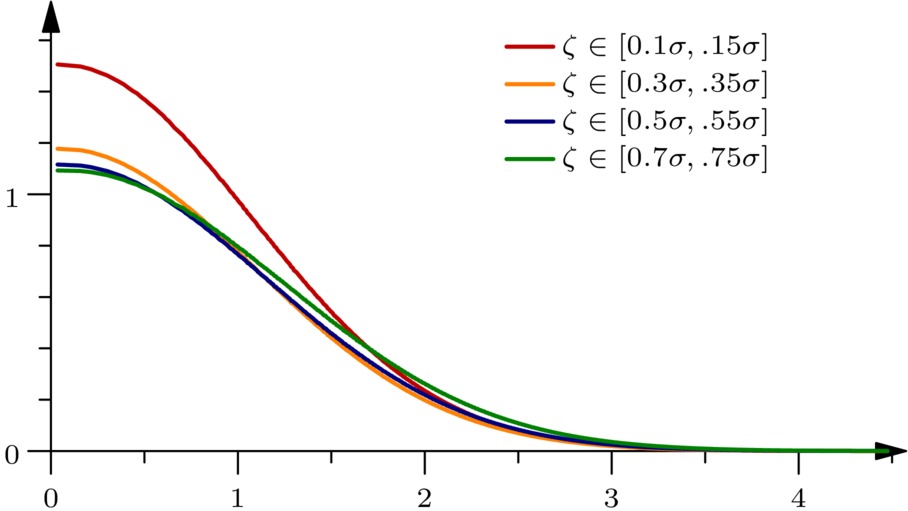}
    \caption{Bump envelopes $\enva{B}$}
    \label{fig:BumpEnvelopes}
\end{subfigure}
\begin{subfigure}{0.5\textwidth}
    \centering
    \includegraphics[scale=0.5]{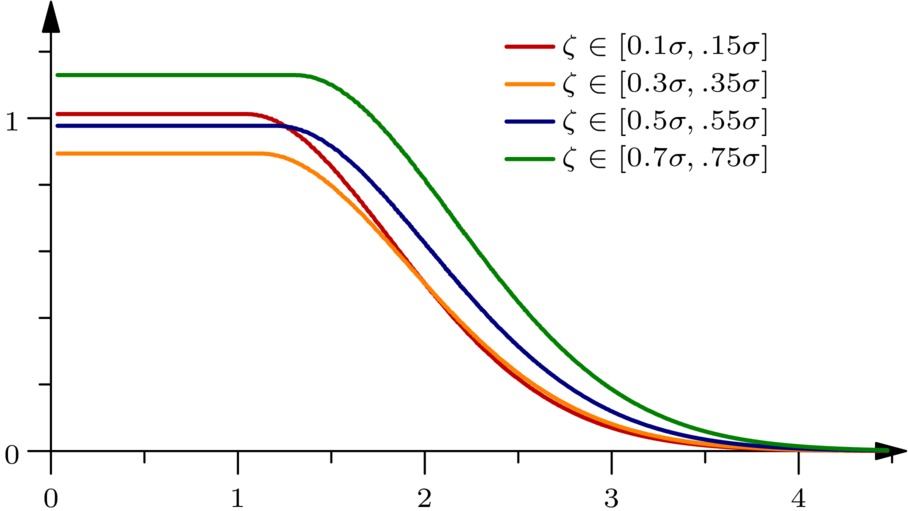}
    \caption{Wave envelopes $\enva{W^1}$}
    \label{fig:WaveEnvelopes}
\end{subfigure}
\caption{\label{fig:Envelopes} Monotonized envelopes for bump and wave functions. Horizontal axis units are in terms of $\sigma$.}
\end{figure}

\subsection{Invertibility of the Interpolation 
Equations}
\label{sec:InvertibilityProof}
To prove \Cref{lem:Invertibility} we show that the reparametrized interpolation equations 
\eqref{eq:BlockEquation} have a unique solution.  
Intuitively, when the minimum separation $\spikesep(T)$ satisfies the conditions of \Cref{thm:Exact}, this system of equations will be approximately diagonal. We formalize this intuition in the following result,
which shows that bounding the norms of the nine block matrices in equation \eqref{eq:BlockEquation} is enough to prove that a unique solution exists, and also yields bounds on the solutions $\bcoeff$, $\wocoeff$, and $\wtcoeff$.  Throughout, for an $n\times n$ matrix $A$, we write $\norminf{A}$ to denote the matrix norm
\begin{equation}
    \norminf{A} = \sup_{\norminf{x}\leq1}\norminf{Ax}.
\end{equation}

\begin{lemma}[Proof in \Cref{sec:SchurInvProof}]\label{lem:SchurInvEasy}
Suppose
\begin{enumerate}
\item $\norminf{\II-\DtWt}<1$,
\item $\norminf{\II-\SSo}<1$, and
\item $\norminf{\II-\SSth}<1$,
\end{enumerate}
where
\begin{align}
\SSo&=\DoWo-\DoWt\DtWti\DtWo,\\
\SSt&=\DoB-\DoWt\DtWti\DtB,\\
\SSth&=\BB-\WWo\SSoi\SSt+\WWt\DtWti(\DtWo\SSoi\SSt-\DtB).
\label{eq:SubMatrix}
\end{align}
Then
\begin{enumerate}
\item Equation \eqref{eq:BlockEquation} has a unique solution,
\item $\norminf{\bcoeff}\leq\norminf{\SSthi},$
\item $\norminf{\wocoeff}, \norminf{\wtcoeff}\leq\norminf{\SSoi}\norminf{\SSt}\norminf{\SSthi}$,
\item $\norminf{\bcoeff-\signs}\leq\norminf{\SSthi}\norminf{I-\SSth}$,
\item $\abs{\bcoeff_i}\geq 1-\norminf{\SSthi}\norminf{I-\SSth}$ for all $i$.
\end{enumerate}
\end{lemma}
The proof of \Cref{lem:SchurInvEasy} follows the same arguments given in Appendix C.1 of \cite{superres} and Lemma 3.9 in \cite{bernstein2018}. 

To complete the proof of \Cref{lem:Invertibility}, we must determine how to bound the matrix norms in \Cref{lem:SchurInvEasy} in terms of the minimum separation $\spikesep(T)$.
Bounding these norms in the two-dimensional case is more challenging than in one dimension.
We explain how to overcome this challenge in \Cref{sec:Geometry}.
Then, in \Cref{sec:Norms} we calculate these norm bounds using the envelope functions \eqref{eq:Envelopes} defined in \Cref{sec:BWBounds}.

\subsection{Spike Distances and Geometric Considerations}\label{sec:Geometry}
In this section we calculate bounds on the norms of the nine block matrices in equation~\eqref{eq:BlockEquation}.  This requires taking into account geometric considerations that do not arise in the one-dimensional setting.

Using the envelopes defined in \Cref{sec:BWBounds}, we have
\begin{align}
    \norminf{\II-\BB}
    &= \max_j \sum_{i\neq j}|B_i(t_j)|\label{eq:matrixBound1}\\
    \label{eq:matrixBound2}
    &\leq \max_j \sum_{i\neq j}\enva{B}(\|t_j-t_i\|),
\end{align}
with analogous formulas for the other blocks.
Since $\enva{B}(r)$ decreases monotonically as $r$ increases, the above bound is maximized when the spikes are clustered. Determining which spike configuration maximizes \eqref{eq:matrixBound2} while satisfying the minimum separation condition appears to be a very difficult task.
Instead, we bound this quantity by applying the pigeonhole principle to a hexagonal tiling of the plane.  

\begin{figure}[t!]\centering
\begin{subfigure}{.45\textwidth}
\centering
\includegraphics[scale=1]{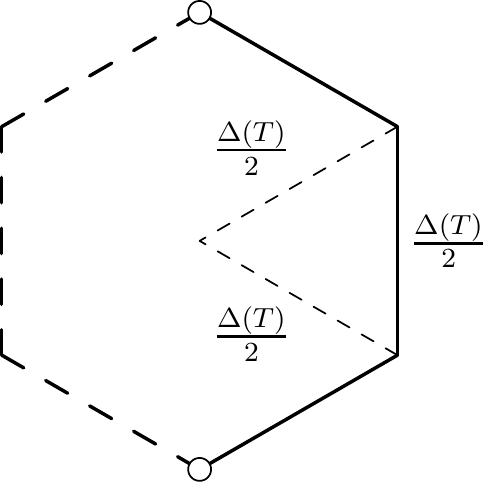}
\end{subfigure}
\begin{subfigure}{.45\textwidth}
\centering
\includegraphics[scale=1]{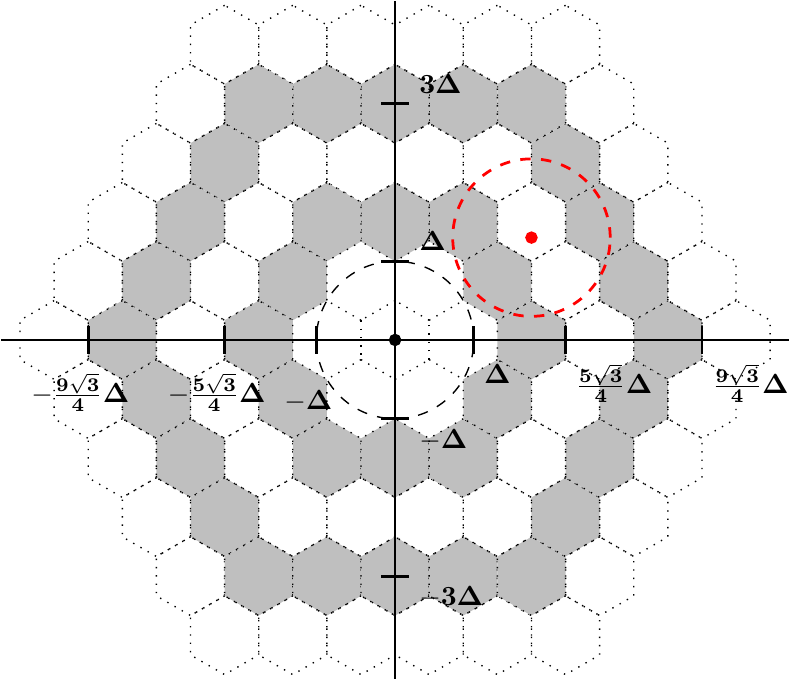}
\end{subfigure}
\caption{Depiction of a hexagonal partition of $\RR^2$.  The left image shows a single regular hexagon with radius and side length $\spikesep(T)/2$.  Note that the left border, top, and bottom points are omitted from the hexagon.  The right image shows a partition of the plane into disjoint hexagons.  The spike at the origin precludes any other spike existing within a disk of radius $\spikesep$.  The red spike and circle of radius $\spikesep(T)$ illustrate that there can be at most one spike per hexagonal cell. The hexagons are divided into layers, depending on their distance to the origin.}
\label{fig:hexgrid}
\end{figure}

For each $t_j\in T$ we consider a coordinate system where $t_j$ is at the origin.  As depicted in \Cref{fig:hexgrid}, we partition $\RR^2$ into a collection of disjoint regular hexagons with sides of length $\spikesep(T)/2$.  Let $\{U_i\}_{i=0}^\infty$ be an enumeration of these hexagons, which we call \textit{cells}, with $U_0$ denoting the cell centered at the origin. Note that if $x,y\in U_i$ for some $i$, then $\|x-y\|<\spikesep(T)$, so each cell contains at most one spike.   
The cells can be arranged in concentric rings, called \textit{layers}, also depicted in \Cref{fig:hexgrid}.
Each layer, shaded with the same color, forms a contiguous ring of hexagons surrounding the previous layer. Layer one is formed by six white inner hexagons that overlap with the central circle, layer two is the ring of twelve gray hexagons surrounding the first layer, and so on.

The next lemma shows that, when bounding the sums in equation~\eqref{eq:matrixBound1}, contributions from spikes in layers nine and higher are negligible. Its proof, given in \Cref{sec:BumpTailProof}, exploits the rapidly decaying tails of the Gaussian kernel.

\begin{restatable}[Proof in \Cref{sec:BumpTailProof}]{lemma}{BumpTail}\label{lem:BumpTail}
Fix $t_j\in T$ at the origin. Let $10^{-2}<\gridsep\leq1$, and $\spikesep(T)\geq2$. Let $\UU_{\geq 9}$ denote the union of all hexagonal cells in layers nine and higher.  If $z$ is any point with $\norm{z}_2\leq\spikesep(T)$, then
\begin{equation}
\sum_{t_k\in T\cap\,\UU_{\geq9}}\abs{f(t_k-z)}< 2\times10^{-12}=:\epsbump\mbox{\quad and\quad}
\sum_{t_k\in T\cap\,\UU_{\geq9}}\abs{g(t_k-z)}< 2\times 10^{-10}=:\epswave,
\end{equation}
where $f$ is the bump $B_j$, or any of its first and second partial derivatives,
and $g$ is the wave $W^i_j$, or any of its first and second partial derivatives, for $i=1,2$.
\end{restatable}

\Cref{lem:BumpTail} shows that distant spikes have negligible contribution to the sums in \eqref{eq:matrixBound1}.  In the following lemma we use this fact to provide bounds on the norms of the nine matrices in \eqref{eq:BlockEquation}.

\begin{lemma}\label{lem:normbound} 
Suppose $\spikesep(T)\geq 2$ and $10^{-2}<\gridsep\leq2$.  Let $\CC_{\leq 8}$ denote the collection of hexagonal cells in the inner eight layers (which excludes the cell $U_0$ containing the origin). We have
\begin{equation}
\begin{gathered}\label{eq:normbounds}
\norminf{\II-\BB}\leq\sum_{U\in\CC_{\leq8}}\enva{B}(d_U)+\epsbump,\\
\norminf{\DoB}\leq\sum_{U\in\CC_{\leq8}}\enva{\partial_xB}(d_U)+\epsbump,\quad
\norminf{\DtB}\leq\sum_{U\in\CC_{\leq8}}\enva{\partial_yB}(d_U)+\epsbump\\
\norminf{\WWo}\leq\sum_{U\in\CC_{\leq8}}\enva{W^1}(d_U)+\epswave,\quad
\norminf{\WWt}\leq\sum_{U\in\CC_{\leq8}}\enva{W^2}(d_U)+\epswave,\\
\norminf{\II-\DoWo}\leq\sum_{U\in\CC_{\leq8}}\enva{\partial_xW^1}(d_U)+\epswave,\quad
\norminf{\DoWt}\leq\sum_{U\in\CC_{\leq8}}\enva{\partial_xW^2}(d_U)+\epswave,\\
\norminf{\DtWo}\leq\sum_{U\in\CC_{\leq8}}\enva{\partial_yW^1}(d_U)+\epswave,\quad
\norminf{\II-\DtWt}\leq\sum_{U\in\CC_{\leq8}}\enva{\partial_yW^2}(d_U)+\epswave,
\end{gathered}
\end{equation}
where $d_U=\inf\{\norm{x}:x\in U,\norm{x}\geq\spikesep(T)\}$,  $\epsbump=2\times10^{-12}$, $\epswave=2\times10^{-10}$.
\end{lemma}
\begin{proof}
Fix $t_j\in T$ and assume, without loss of generality, that it is positioned at the origin.
Let $S$ denote the set of all spikes (excluding $t_j$) in the first eight layers closest to the origin. Then we have
\begin{align}
\sum_{t_i:t_i\neq t_j}\abs{B_i(t_j)}
&\leq\sum_{t_i\in S}\abs{B_i(t_j)}+\epsbump\\
&\leq\sum_{t_i\in S}\enva{B}(\norm{t_i-t_j})+\epsbump\\
&\leq\sum_{U\in\CC_{\leq8}}\enva{B}(d_U)+\epsbump\label{eq:mono}
\end{align}
where the first inequality uses \Cref{lem:BumpTail}, the second follows by the definition of the envelope $\enva{B}$, and the last by the monotonicity of $\enva{B}$ and the definition of $d_U$.
This bound applies to any spike $t_j$, so we have that $\norminf{\II-\BB}$ is less than the value in \eqref{eq:mono}. The other eight norm bounds are derived in the same way.
\end{proof}

For brevity, we omit the simple yet tedious calculation of the $d_U$-values required by \Cref{lem:normbound}.  The locations that determine the $d_U$-values for the four innermost layers are depicted in \Cref{fig:spikegrid}.
In the next section we combine \Cref{lem:normbound} with \Cref{lem:SchurInvEasy} to prove \Cref{lem:Invertibility}.

\begin{figure}[!t]
\centering
\includegraphics{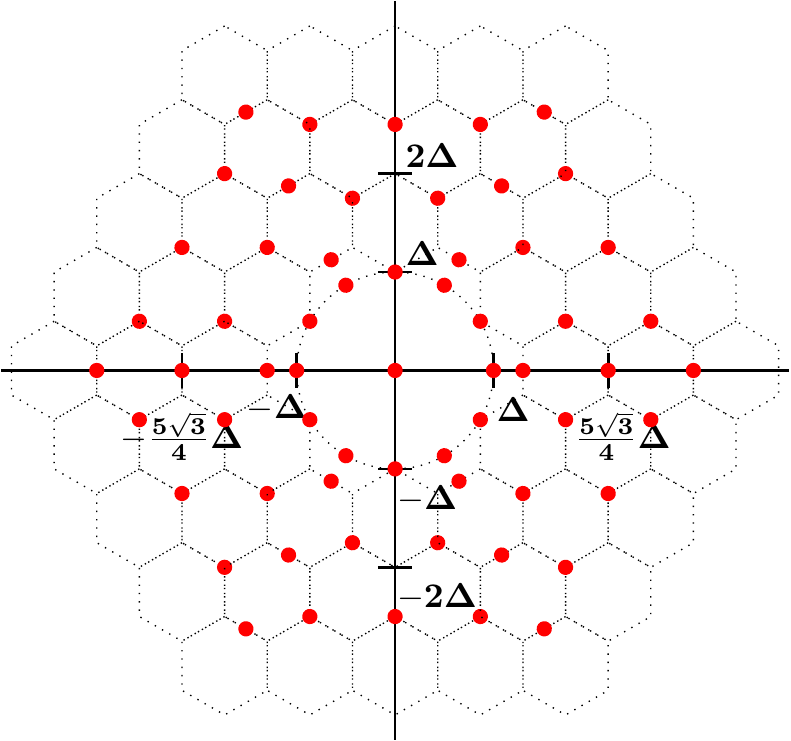}
\caption[Nearest Hexagonal Grid Points ]{The red dots indicate the locations of the closest points to the origin with norm at least $\spikesep(T)$ that belong to the each of the four innermost layers.
The norms of these points determine 
the values of $d_U$ in \Cref{lem:normbound}.}\label{fig:spikegrid}
\end{figure}

\subsection{Proof of \Cref{lem:Invertibility}}\label{sec:Norms}

\begin{figure}[!t]
\centering
\includegraphics[scale=0.5]{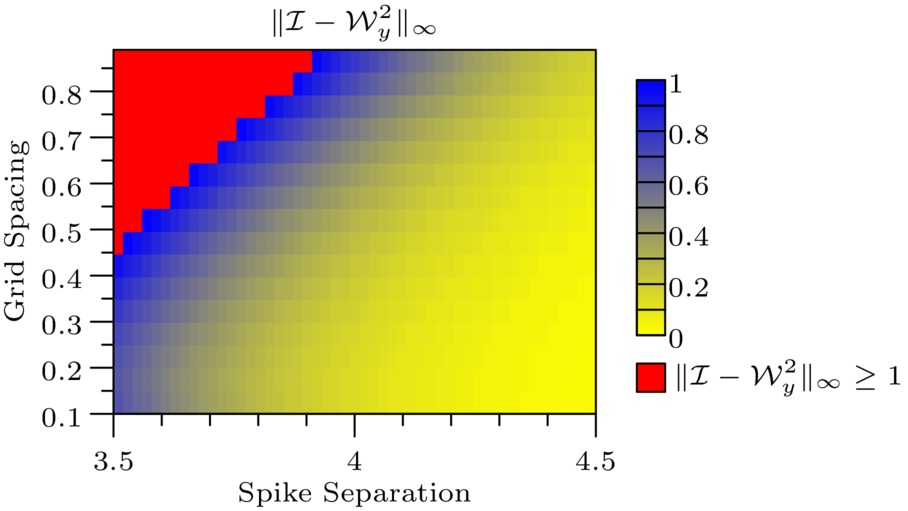}
\includegraphics[scale=0.5]{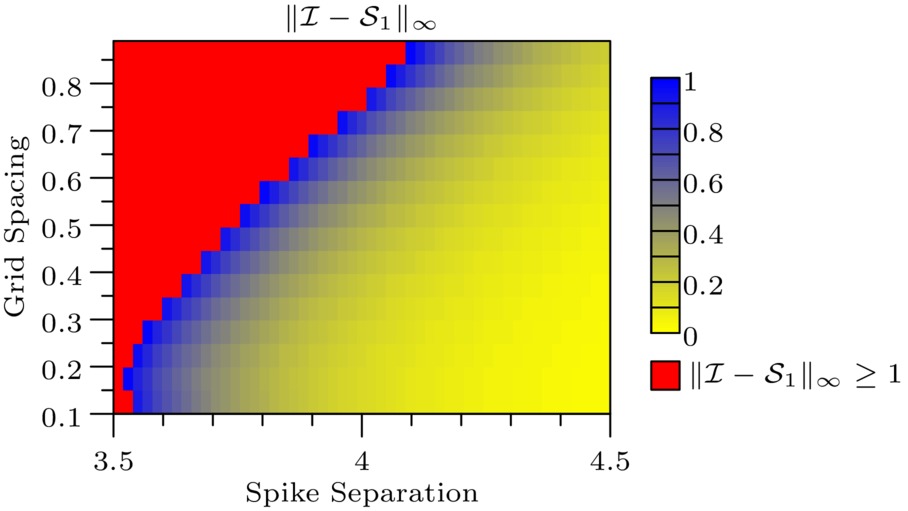}
\includegraphics[scale=0.5]{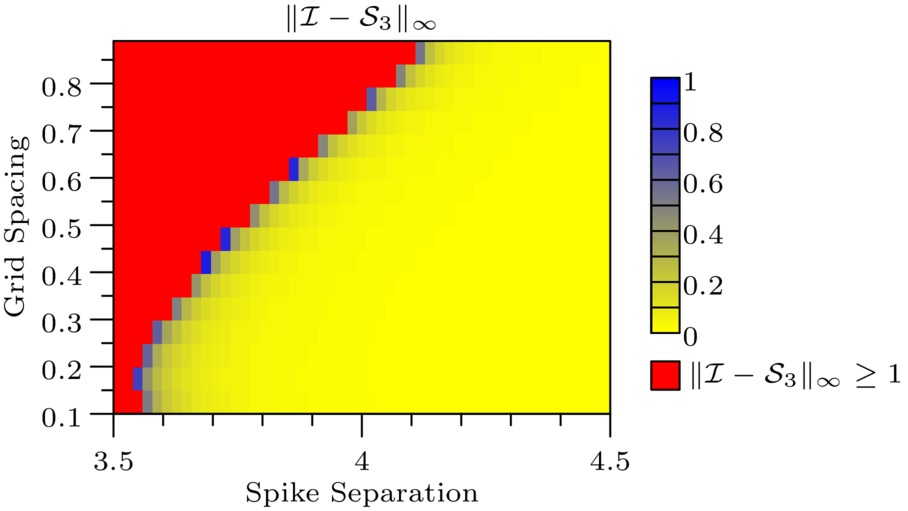}
\includegraphics[scale=0.5]{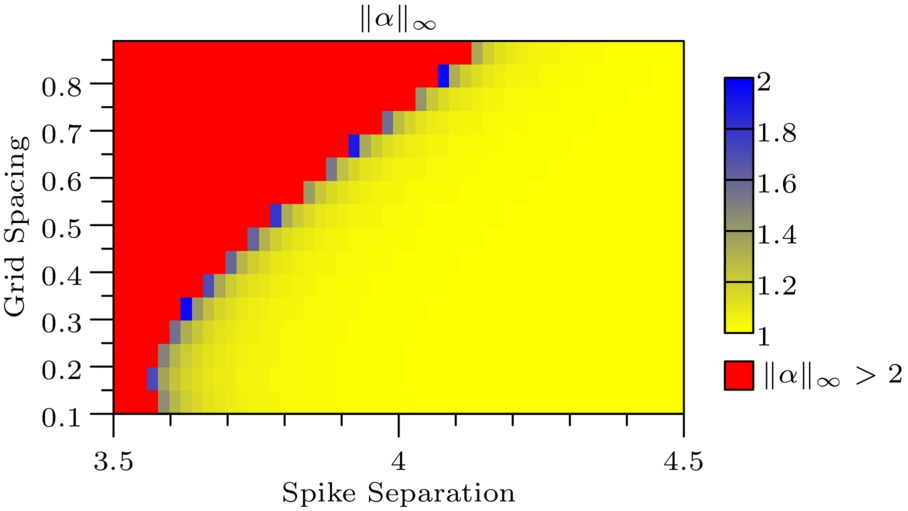}
\includegraphics[scale=0.5]{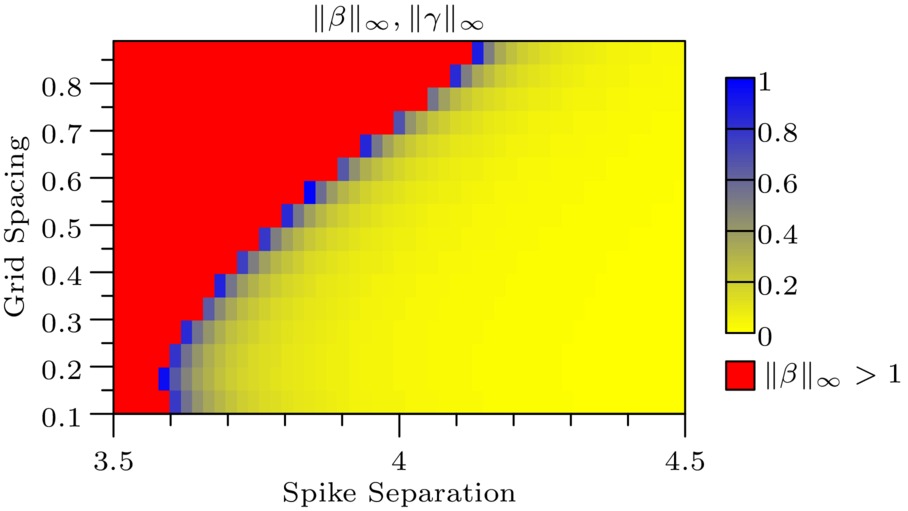}
\includegraphics[scale=0.5]{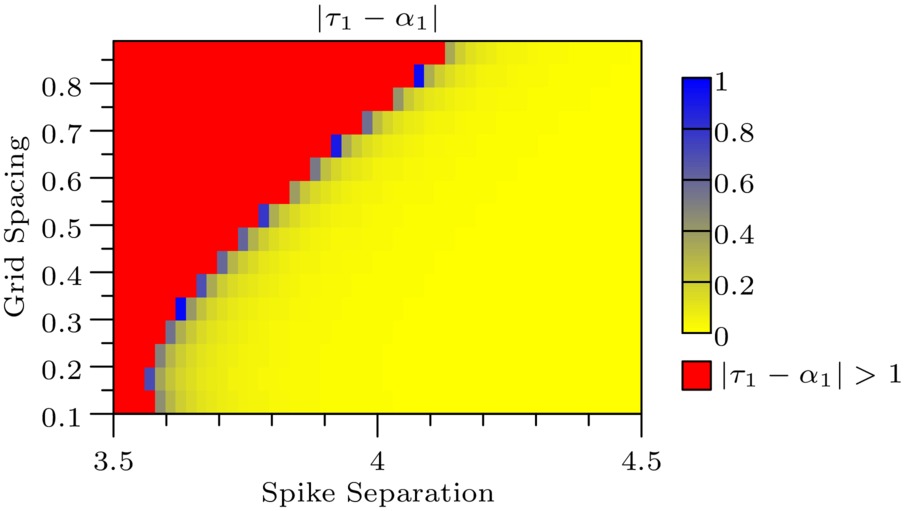}
\caption[Plots of Values of Coefficients and Matrix Norms]{The graphs show upper bounds on the matrix norms in \Cref{lem:SchurInvEasy}, and on the coefficient vectors $\bcoeff$, $\wocoeff$ and $\wtcoeff$ for a range of spike separations and grid spacings. In the red regions the bounds are too large to prove that \eqref{eq:conditionQTc} holds.}
\label{fig:NormsCoeffs}
\end{figure}

\Cref{lem:normbound} provides norm bounds on the nine block matrices defined in \eqref{eq:BlockEquation}.  The bounds are computed by evaluating the envelope functions defined in 
\Cref{sec:BWBounds} (and calculated in \Cref{sec:Envelopes}) at each $d_U$-value.  To satisfy the conditions of \Cref{lem:SchurInvEasy}, we must prove that
\begin{equation}
    \norminf{\II-\DtWt}<1,\quad
    \norminf{\II-\SSo}<1, \quad\text{and}\quad
    \norminf{\II-\SSth}<1,
\end{equation}
where $\SSo$ and $\SSth$ are defined in \Cref{lem:SchurInvEasy}.  Note that
\begin{align}
    \norminf{\II-\SSo}
    & = \norminf{\II-\DoWo+\DoWt\DtWti\DtWo}\\
    & \leq \norminf{\II-\DoWo}+\norminf{\DoWt\DtWti\DtWo}
    & \text{(Triangle Inequality)}\\
    & \leq \norminf{\II-\DoWo}+\norminf{\DoWt}
    \norminf{\DtWti}\norminf{\DtWo} & \text{(Sub-multiplicativity)}\\
    & \leq \norminf{\II-\DoWo}+\frac{\norminf{\DoWt}
    \norminf{\DtWo}}{1-\norminf{\II-\DtWt}},
\end{align}
where the last inequality follows from
\begin{equation}
    \norminf{\DtWti} \leq \frac{1}{1-\norminf{\II-\DtWti}},
\end{equation}
since $\norminf{\II-\DtWt}<1$ (see the proof of \Cref{lem:SchurInv} for more details).  Thus we have bounded $\norminf{\II-\SSo}$ in terms of quantities computed in \Cref{lem:normbound}.  By similar logic we can also bound $\norminf{\II-\SSth}$ in terms of the matrix bounds computed in \Cref{lem:normbound}.  In \Cref{fig:NormsCoeffs} we compute these bounds over a range of $(\gridsep,\spikesep(T))$ pairs.
The plots are divided into colored rectangles, each representing a length $0.05$ interval of grid separation values, and a length $0.01$ interval of spike separation values.  The color displayed in each rectangle is the corresponding upper bound that applies to all $(\gridsep,\spikesep(T))$-values in that rectangle.  
As required, the region where the parameters $(\spikesep(T),\gridsep)$ satisfy the assumptions of \Cref{thm:Exact} is a subset of the region where 
the assumptions of \Cref{lem:SchurInvEasy} are satisfied.
This guarantees the existence of a solution $(\bcoeff, \wocoeff,\wtcoeff)$ to \eqref{eq:BlockEquation} which in turn yields an interpolation function $Q$ defined by \eqref{eq:combination} that proves \Cref{lem:Invertibility}. As a byproduct, we obtain bounds on the solutions $\bcoeff$, $\wocoeff$ and $\wtcoeff$ that are used in the following sections.

\subsection{Bounding the Interpolation Function (Proof of \Cref{lem:QBound})}\label{sec:qboundproof}

In \Cref{lem:Invertibility} we prove the existence of an interpolating function $Q$ 
when the assumptions of \Cref{thm:Exact} hold.
Here we show $\abs{Q}<1$ on $T^c$ proving \Cref{lem:QBound}.  For points in $T^c$ close to an element of $T$, we exploit the curvature and slope of the interpolating function to establish the bound.  For points distant from $T$, we upper bound the magnitude of $Q$ directly. 

Fix $t\in T^c$.  We first assume $\|t-t_j\|\geq \spikesep(T)$ for all $t_j\in T$.   Using \eqref{eq:combination} we can decompose $Q$ into a sum of bumps and waves:
\begin{equation}
    Q(t) = \sum_{t_j\in T} \bcoeff_jB_j(t)+\wocoeff_jW_j^1(t)+\wtcoeff_jW_j^2(t)
\end{equation}
Noting that $T'=T\cup\{t\}$ has minimum separation $\spikesep(T)$, we bound $Q(t)$ using the matrix norms in \Cref{lem:normbound} replacing $T$ by $T'$:
\begin{align}
    |Q(t)| & \leq \sum_{t_j\in T} |\bcoeff_j||B_j(t)|+|\wocoeff_j||W_j^1(t)|+|\wtcoeff_j||W_j^2(t)|\\
    & \leq \norminf{\bcoeff}\sum_{t_j\in T'\setminus\{t\}}|B_j(t)| + 
    \norminf{\wocoeff}\sum_{t_j\in T'\setminus\{t\}}|W_j^1(t)| + 
    \norminf{\wtcoeff}\sum_{t_j\in T'\setminus\{t\}}|W_j^2(t)|\\
    & \leq \norminf{\bcoeff}\norminf{I-\wt{\BB}}+\norminf{\wocoeff}\norminf{\wt{\WWo}}
    +\norminf{\wtcoeff}\norminf{\wt{\WWt}},
\end{align}
where $\wt{\BB}$, $\wt{\WWo}$, and $\wt{\WWt}$ are the submatrices corresponding to the enlargened spike set $T'$, and the coefficient bounds are seen in \Cref{fig:NormsCoeffs}.
The values of these norms are less than the bounds determined in \Cref{lem:normbound}.
In \Cref{fig:QRecoveryDistant} we compute this bound over a range of $(\spikesep(T),\gridsep)$ pairs and show that it is strictly less than 1 when the assumptions of \Cref{thm:Exact} are satisfied.  This completes the proof when $\norm{t-t_j}\geq\spikesep(T)$ for all $t_j\in T$.

\begin{figure}[t!]
    \centering
    \includegraphics[scale=0.5]{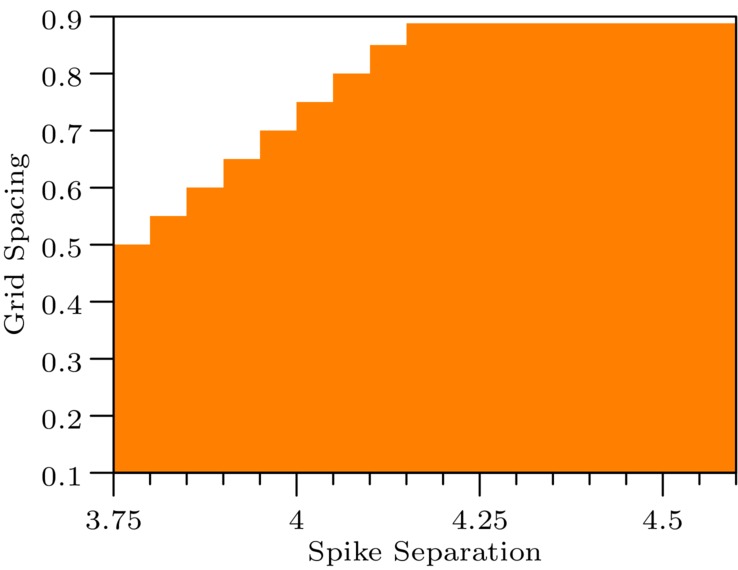}
    \caption{$(\spikesep(T),\gridsep)$ pairs where $\abs{Q(t)}<1$ when $\norm{t-t_j}\geq\spikesep(T)$ for all $t_j\in T$.}
    \label{fig:QRecoveryDistant}
\end{figure}

Next we handle the case when $\norm{t-t_j}<\spikesep(T)$ for some $t_j\in T$.  The following lemma from calculus, applied with $h:=Q$ and $L:=\spikesep(T)$, allows us to use derivative information to produce sharper bounds on $Q$.

\begin{lemma}\label{lem:Regions}
Assume that $h:\RR^n\to\RR$ has continuous second partial derivatives, and satisfies $h(x_0)=1$ and $\nabla h(x_0)=0$ for some $x_0\in\RR^n$. If there are $u_1$, $u_2\in(0,L]$ 
such that for all unit vectors $v$
\begin{enumerate}
\item $\int_{0}^{r}(v^T\nabla^2 h(x_0+sv)v)(r-s)\,ds<0$ for $r\in(0,u_1]$,
\item $\int_{0}^{u_1}(v^T\nabla^2 h(x_0+sv)v)(r-s)\,ds+\int_{u_1}^{r}\nabla h(x_0+sv)^T v\,ds<0$ for $r\in(u_1,u_2]$,
\item $|h(x_0+sv)|<1$ for $s\in[u_2,L]$, and
\item $h(x_0+sv)>-1$ for $s\in(0,L]$,
\end{enumerate}
then $|h(x)|<1$ for all $x$ with $0<\norm{x-x_0}\leq L$.
\end{lemma}
\begin{proof}
Fix $v$ with $\|v\|=1$ and define $g:\RR\to\RR$ by $g(s)=h(x_0+sv)$.  For fixed $r>0$ define $x=x_0+rv$ and note that
\begin{align}
h(x)-h(x_0)&=g(r)-g(a)+g(a)-g(0)\\
&=\int_{a}^{r}g^\prime(s)\,ds+\int_{0}^{a}g'(s)\,ds\\
&=\int_{a}^{r}g^\prime(s)\,ds+\int_{0}^{a}g^{\prime\prime}(s)(a-s)\,ds\\
&=\int_{a}^{r}\nabla h(x_0+sv)^T v\,ds+\int_{0}^{a}(v^T\nabla^2 h(x_0+sv)v)(a-s)\,ds,\label{eq:FTC}
\end{align}
for any $a$, where the second equality follows by the fundamental theorem of calculus, and the third by integration by parts, and the assumption that $g'(0)=\nabla h(x_0)^Tv=0$.
The result follows by letting $a=r$ for $r\in(0,u_1]$, and $a=u_1$ for $r\in[u_2,L]$.
\end{proof}

\begin{figure}[t!]
\begin{subfigure}{0.5\textwidth}
    \centering
    \includegraphics[scale=1.4]{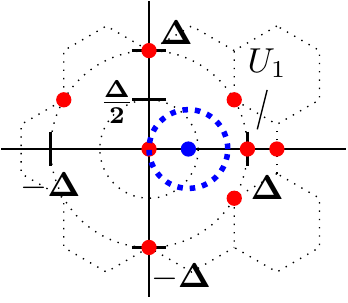}
    \caption{}\label{fig:nearpoint}
\end{subfigure}
\begin{subfigure}{0.5\textwidth}
    \centering
    \includegraphics[scale=1.4]{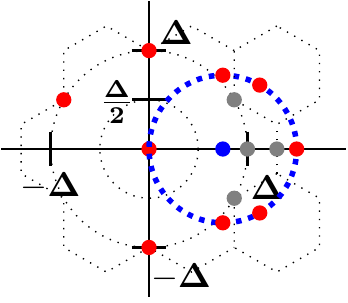}
    \caption{}\label{fig:distantpoint}
\end{subfigure}
    \caption{In both figures, the point $t\in T^c$, colored blue, is closest to the red spike at the origin.  In \Cref{fig:nearpoint} the distance $\norm{t}$ is less than $\spikesep(T)/2$, so the closest possible spikes in each hexagonal cell are only constrained by the minimum separation condition.  In contrast, \Cref{fig:distantpoint} shows that as $\norm{t}$ approaches $\spikesep(T)$, the constraint that the spike at the origin is closest to $t$ pushes the other spikes away from the boundary (the unconstrained points are shown in grey). }
    \label{fig:neardistantpoint}
\end{figure}

Let $t_1$ denote the point in $T$ that is closest to $t$.  
Without loss of generality, we can assume that $t_1$ is at the origin, that $Q(t_1)=1$, and that $t$ lies on the positive horizontal axis.
Although the samples and partial derivatives may not be axis-aligned in this rotated coordinate system, this does not affect our argument as the envelopes in \Cref{sec:BWBounds} are  radially symmetric and thus invariant under rotation.

As shown in \Cref{fig:neardistantpoint}, if $t$ lies in the interval $[a,b]$ on the positive horizontal axis with $0<a\leq b\leq\spikesep(T)$ then all spikes must have distance at least $a$ from $t$ since the spike $t_1$ at the origin is closest to $t$.  Combining this fact with the coefficient bounds in \Cref{lem:SchurInvEasy}, and the envelopes from \Cref{sec:BWBounds}, we obtain the following result bounding $Q$ and its derivatives.  We define the distance $d(A,B)$ between two sets $A,B\subseteq\RR^2$ by
\begin{equation}
    d(A,B) := \inf_{\substack{x\in A\\y\in B}}\|x-y\|.
\end{equation}

\begin{restatable}[Proof in \Cref{sec:QtriProof}]{lemma}{Qtri}
\label{lem:Qtri}
Assume the conditions of \Cref{thm:Exact} hold, and let $Q$ be the interpolation function constructed in the proof of \Cref{lem:Invertibility}. Furthermore, assume there is a spike $t_1$ at the origin with $Q(t_1)=1$.  Let $[a,b]$ denote an interval on the positive horizontal axis with $0< a\leq b\leq \Delta(T)$, and let $\CC_{\leq8}$ denote the collection of hexagonal cells in the inner eight layers (which excludes the cell $U_0$ containing the origin). For $U\in \CC_{\leq 8}$, define $d_U:=\max(d([a,b],U),a)$. Let $\bcoeff_{\LB}:=1-\norminf{\bcoeff-\signs}\geq0$ denote our lower bound on the coordinates of $\bcoeff$ from \Cref{lem:SchurInvEasy}, and let $\epsilon=10^{-9}$. Then by the envelope function definitions \eqref{eq:Envelopes} and \eqref{eq:RawEnvelopes}
we have, for all $t\in [a,b]$,
\begin{align}
\begin{split}
    |Q(t)| \leq&\norminf{\bcoeff}\enva{B}(a)+\norminf{\wocoeff}\enva{W^1}(a)+\norminf{\wtcoeff}\enva{W^2}(a)\\
    &+\sum_{U\in \CC_{\leq 8}}\norminf{\bcoeff}\enva{B}(d_U)+\norminf{\wocoeff}\enva{W^1}(d_U)+\norminf{\wtcoeff}\enva{W^2}(d_U)+\epsilon,
\end{split}\label{eq:QBound}\\
\begin{split}
    Q(t)\geq&-\norminf{\wocoeff}\enva{W^1}(a)-\norminf{\wtcoeff}\enva{W^2}(a)\\
    &-\sum_{U\in \CC_{\leq 8}}\left(\norminf{\bcoeff}\enva{B}(d_U)+\norminf{\wocoeff}\enva{W^1}(d_U)+\norminf{\wtcoeff}\enva{W^2}(d_U)\right)-\epsilon.
\end{split}\label{eq:QLowerBound}
\end{align}
If we also have $\omega:=\sup_{t\in[a,b]}\env{D (B)}(\|t\|)$ then, for all $t\in[a,b]$,
\begin{equation}
\begin{aligned}\label{eq:RawGradBound}
    \nabla Q(t)\cdot\hat{t} \leq&~\max(\bcoeff_{\LB}\ \omega,\norminf{\bcoeff}\omega) +\norminf{\wocoeff}\norm{\nabla W^1}_{\downarrow}(a)+\norminf{\wtcoeff}\norm{\nabla W^2}_{\downarrow}(a)\\
    &+\sum_{U\in \CC_{\leq 8}}\norminf{\bcoeff}\norm{\nabla B}_{\downarrow}(d_U)+\norminf{\wocoeff}\norm{\nabla W^1}_{\downarrow}(d_U)+\norminf{\wtcoeff}\norm{\nabla W^2}_{\downarrow}(d_U)+\epsilon,
\end{aligned}
\end{equation}
where $\hat{t}=t/\norm{t}$ and
\begin{equation}
    \envan{\nabla B}(s):=(\enva{\partial_x B}(s)^2+\enva{\partial_y B}(s)^2)^{1/2}
\end{equation}
with analogous definitions for $\envan{\nabla W^1}$ and $\envan{\nabla W^2}$.
If we also have $\eta:=\sup_{t\in[a,b]}\env{\lambda(B)}(\|t\|)$ then, for all $t\in[a,b]$,
\begin{align}
\begin{split}\label{eq:RawEigBound}
    \lambda(Q)(t) \leq&~\max(\bcoeff_{\LB}\ \eta,\norminf{\bcoeff}\eta) +\norminf{\wocoeff}\enva{\lambda(W^1)}(a)+\norminf{\wtcoeff}\enva{\lambda(W^2)}(a)\\
    &+\sum_{U\in \CC_{\leq 8}}\norminf{\bcoeff}\enva{\lambda(B)}(d_U)+\norminf{\wocoeff}\enva{\lambda(W^1)}(d_U)+\norminf{\wtcoeff}\enva{\lambda(W^2)}(d_U)+\epsilon
\end{split}
\end{align}
where
\begin{equation}
    \lambda(Q)(t) := \sup_{\|v\|=1} v^T\nabla^2 Q(t) v.
\end{equation}
\end{restatable}

\begin{figure}[t]
    \centering
    \includegraphics{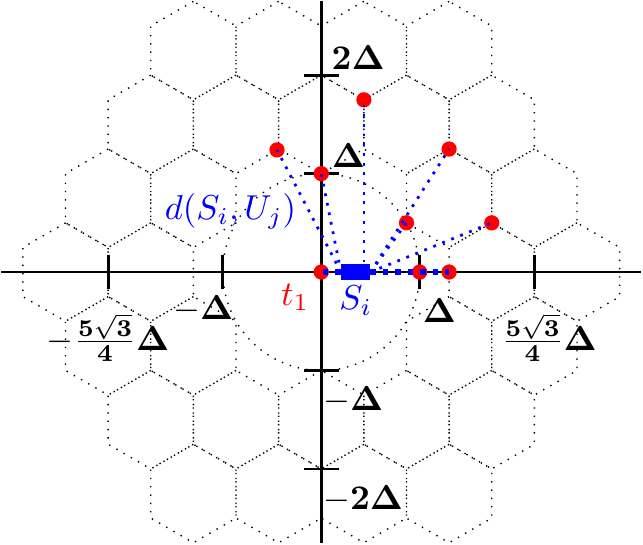}
    \caption{The distances $d(S_i,U)$ for a fixed interval $S_i$ and several choices of $U\in\CC_{\leq 8}$.  These distances are used in \Cref{lem:Qtri} to compute bounds on the interpolation function $Q$.}
    \label{fig:intervalcelldists}
\end{figure}

By repeatedly applying \Cref{lem:Qtri} we obtain the following procedure for establishing exact recovery for a fixed minimum separation $\spikesep(T)$, and fixed range of grid spacings $[\gridsep_1,\gridsep_2]$:
\begin{enumerate}
    \item Partition the interval $[0,\spikesep(T)]$ of the positive horizontal axis into 100 segments of equal length $S_1,\ldots,S_{100}$ where $S_i:=((i-1)\Delta/100,i\Delta/100]$.
    \item Apply \Cref{lem:Qtri} to obtain bounds on $Q$ and its derivatives over each segment $S_i$. 
    Note that the distances $d(S_i,U)$ (depicted in \Cref{fig:intervalcelldists}) for $U\in \CC_{\leq 8}$ can be precomputed for  $\spikesep(T)=1$.
    The distances for other values of $\spikesep(T)$ are then obtained through dilation.
    \item Use the bounds computed in the previous part to determine if there are choices of $u_1$ and $u_2$ that satisfy \Cref{lem:Regions}.
    \eqref{eq:QBound}, \eqref{eq:RawGradBound}, and \eqref{eq:RawEigBound} are used to bound $Q<1$, while \eqref{eq:QLowerBound} guarantees $Q>-1$ up to $u_2$.
    Note that this can be done efficiently since our bounds on $Q$ and its derivatives are constant on each segment.
    \item If $u_1$ and $u_2$ exist, report success.
\end{enumerate}
\begin{figure}[!t]
\begin{subfigure}{0.5\textwidth}
\includegraphics[scale=0.5]{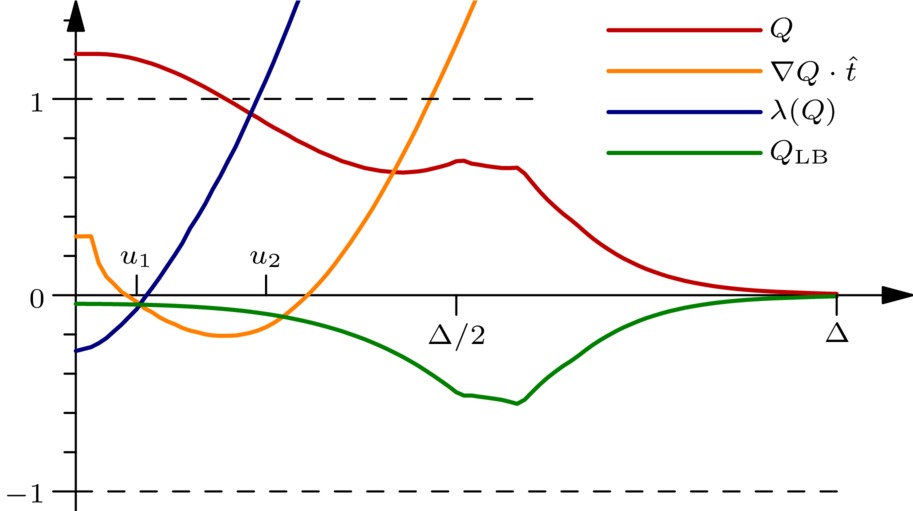}
\caption{\label{fig:QRecovery}}
\end{subfigure}
\begin{subfigure}{0.5\textwidth}
\centering
\includegraphics[scale=0.5]{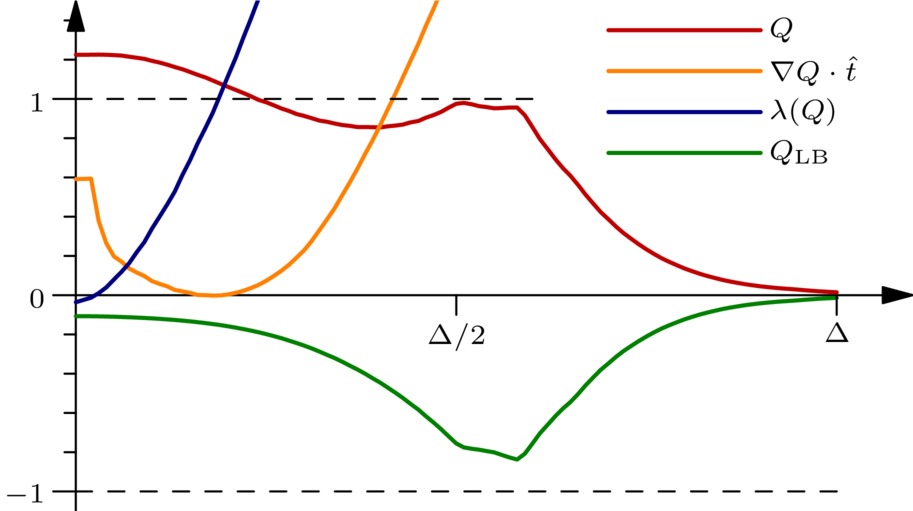}
\caption{\label{fig:QNoRecovery}}
\end{subfigure}
\caption{\label{fig:RecoveryExamples} \Cref{fig:QRecovery} illustrates the computations proving recovery when $\spikesep=4.15\sigma$ and $\gridsep\in[0.3\sigma,0.35\sigma]$. \Cref{fig:QNoRecovery} shows an example where recovery could not be proven when $\spikesep=4.2\sigma$ and $\gridsep\in[0.6\sigma,0.65\sigma]$.}
\end{figure}
The above procedure is applied using Mathematica\footnote{The code is available online at \url{https://github.com/jpmcd/Deconvolution2D}.} to resolve \Cref{lem:QBound} and establish the exact recovery region in \Cref{thm:Exact}.  
In \Cref{fig:RecoveryExamples} we illustrate this computation for two choices of minimum separation and grid spacing values, showing a case where recovery is possible and another where recovery cannot be proven. The curve $Q_{\LB}$ depicts the lower bound for $Q$ in \eqref{eq:QLowerBound}.
In \Cref{fig:QRecovery} the region where $\lambda(Q)(t)$ is negative overlaps with the region where $\nabla Q\cdot\hat{t}$ is negative which extends into the region where $\abs{Q}<1$. Possible choices for $u_1$ and $u_2$ are shown.

\section{Numerical Results}\label{sec:Numerical}

\subsection{Conditioning of Convolution Measurements with Fixed Support}\label{sec:SVD}



In stark contrast to compressed sensing, where randomized measurements preserve the norm of sparse vectors with high probability, the problem of deconvolving sparse signals can be ill-posed. Signals with clustered supports may yield essentially indistinguishable measurements after being convolved with a Gaussian kernel. Suppose $\tilde{x}=\sum_{t_j\in T}\tilde{a}_j\delta_{t_j}$ represents the difference between two signals where $\norm{\tilde{a}}_2=1$.
The observable difference in the signals measured at sample points $s_i$ is given by $\sum_{t_j\in T} \tilde{a_j}K(s_i-t_j)$.
If $\KK$ is the matrix with entries $K(s_i-t_j)$ then the $\ell_2$ norm of this difference, $\|\KK\tilde{a}\|_2$, ranges between the largest and smallest singular values of $\KK$.
In a noisy setting, if the noise is comparable to the size of the smallest singular values then the measured difference between the two signals can be completely corrupted by noise.
The two signals would then produce indistinguishable measurements. In order to characterize when the problem is ill-posed, we compute the singular values of $\KK$ numerically for signals with different separations. 

\begin{figure}[!t]
\begin{subfigure}{0.5\textwidth}
\centering
\includegraphics[scale=0.5]{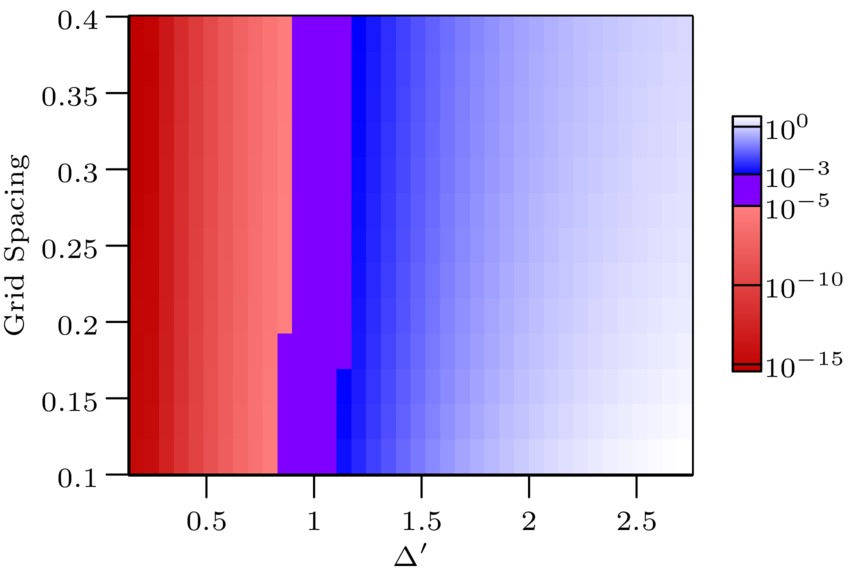}
\caption{\label{fig:gaussSingSepVsSepMin}Smallest singular value.}
\end{subfigure}
\begin{subfigure}{0.5\textwidth}
\includegraphics[scale=0.5]{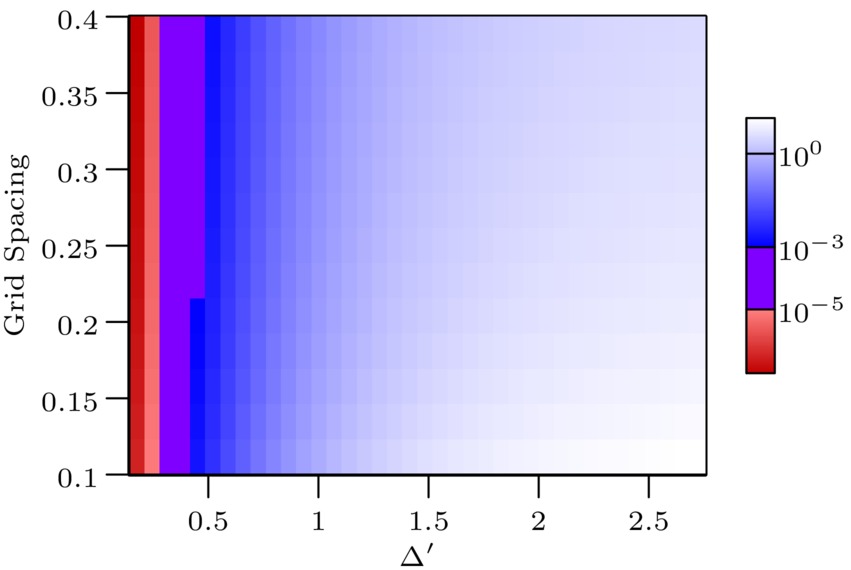}
\caption{\label{fig:gaussSingSepVsSepMed}Middle singular value.}
\end{subfigure}
\caption[Singular Values of the Convolution Kernel]{\label{fig:gaussSingSepVsSep}\Cref{fig:gaussSingSepVsSepMin} shows the value of the smallest singular value of the kernel matrix defined in \Cref{sec:SVD}. \Cref{fig:gaussSingSepVsSepMed} shows the value of the middle singular value. The units for both axes are in terms of $\sigma$.}
\end{figure}

In more detail, we fix an $8\times8$ grid of points $T=\{t_1,\ldots,t_{64}\}\in\RR^2$ with a separation of $\spikesep'$, and a square grid of samples $s_i$ separated by a fixed grid spacing $\gridsep$.
$T$ can be interpreted as the support of the difference between two signals.
We compute the singular values of $\KK$ for different values of $\spikesep'$. 
The smallest singular value corresponds to the smallest observable difference between any two signals with difference of size 1, a worst-case scenario.
The middle singular value quantifies the average observable difference.
We plot the smallest and middle singular values for different values of $\spikesep'$ and $\gridsep$ in \Cref{fig:gaussSingSepVsSep}.
For the smallest singular value a transition occurs for $\spikesep'$ around $\sigma$. Beyond that point its value diminishes dramatically as $\spikesep'$ decreases. The middle singular value reaches a similar transition at $\spikesep'$ equal to $0.5\sigma$.
The grid spacing does not noticeably affect the conditioning of $\KK$.
These results show that the deconvolution problem is ill posed for classes of signals clustered enough to allow for the minimum separation of their difference to be below $\sigma$, and essentially hopeless when the minimum separation fo their difference is below $0.5\sigma$. 

\subsection{Numerical Recovery of Signals in Two Dimensions}\label{sec:ExactSimulation}

In this section, we evaluate the numerical performance of convex programming for deconvolution in two dimensions. 
 We simulate signals consisting of 25 spikes with amplitudes sampled independently at random from a standard Gaussian distribution. The spikes are positioned on a hexagonal grid with separation $\spikesep$. We set the standard deviation of the Gaussian convolution kernel to $\sigma=1$. Samples of the convolution between the signals and the Gaussian kernel are measured on a square grid with separation $\gridsep$. Recovery is performed by solving Problem \eqref{pr:TVnormDisc} using CVX, a popular convex optimization library~\cite{CVX}. The recovery rate is the fraction of signals for which the $\ell_2$ norm of the difference between the estimated and true signal is below a small tolerance ($10^{-3}$).



\begin{figure}[!t]
\begin{subfigure}{.49\textwidth}
\centering
\includegraphics[scale=0.5]{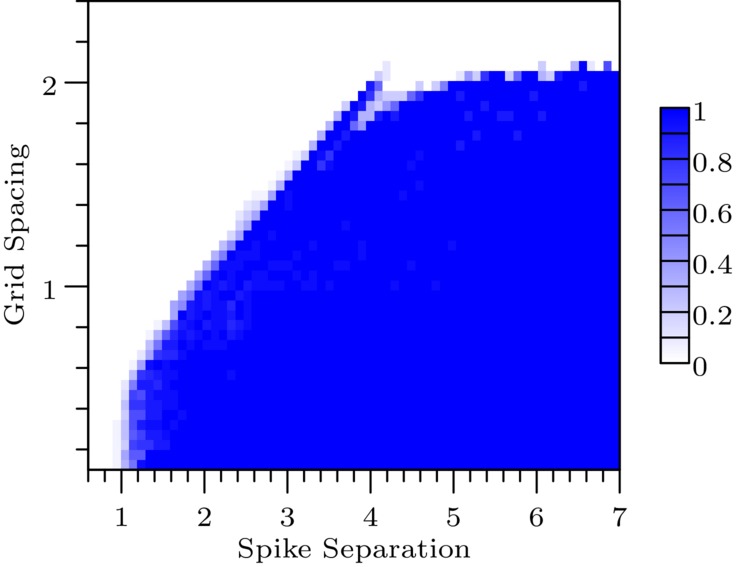}
\caption{\label{fig:ExactCVXFine}}
\end{subfigure}
\begin{subfigure}{.49\textwidth}
\centering
\includegraphics[scale=0.5]{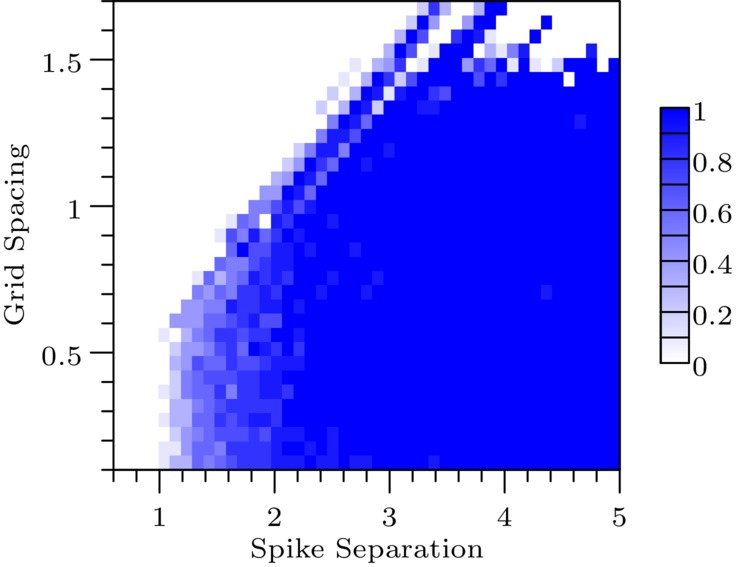}
\caption{\label{fig:ExactCVXLtd}}
\end{subfigure}
\centering
\caption[Numerical Recovery Simulations]{\label{fig:ExactCVX} 
The plots show the results of the numerical experiments described in \Cref{sec:ExactSimulation}. \Cref{fig:ExactCVXFine} shows the fraction of exact recovery for different values of the grid spacing and spike separation. \Cref{fig:ExactCVXLtd} shows results for an artificial sampling pattern inspired by our proof technique, where the measurements consist of three samples close to each spike.}
\end{figure}

The results are shown in \Cref{fig:ExactCVX}. The region of exact recovery is larger than the region in our theoretical guarantees (see \Cref{fig:RecoveryFigures}). When the grid spacing ranges between $0.1\sigma$ and $0.5\sigma$, exact recovery occurs for spike separations greater than roughly $1.4\sigma$. 
The grid spacing of the samples can be as large as $2\sigma$, as long as the spikes are separated enough.

In the proof of our theoretical results, we only use the three nearest samples to each spike. 
In order to evaluate to what extent this may artificially limit the results of the analysis, in \Cref{fig:ExactCVXLtd} we present results for an articial sampling pattern that only contains three samples close to each spike.
This results in a smaller recovery region with fuzzier borders. Exact recovery occurs beyond a spike separation of approximately $2\sigma$ for small grid spacings, and up to a grid spacing of about $1.5 \sigma$ for large spike separations. 



\subsection{Simulations with Convolution Kernels from Microscopy and Telescopy}


In this section we report numerical simulations with point spread functions from two application areas: microscopy and telescopy imaging. Our aim is to show that deconvolution via convex programming yields similar results for these kernels behave as for the Gaussian kernel that is the subject of our theoretical analysis. 
For microscopy, we follow \cite{zhu2012}, which proposed applying $l_1$-norm minimization to perform deconvolution in the context of fluorescence microscopy. The authors experimentally measure the point spread function of the microscope, and find that the radial profile is well fit by a centered Gaussian with smaller off-center ridges, as illustrated on the left of \Cref{fig:ExactCVXMicro}. The precise expression for the point-spread function is 
\begin{equation}
K(x) = e^{-\frac{2\norm{x-x_0}^2}{(1.72)^2}}+0.0208 e^{- \frac{2(\norm{x-x_0}-2.45)^2}{(1.10)^2}}.
\label{eq:MicroKernel}
\end{equation}
For telescopy, we consider a popular model for the point-spread function, the Airy kernel~\cite{goodmanOptics}:
\begin{equation}\label{eq:TeleKernel}
K(x)=\left(\frac{2J_1(3.8317\norm{x})}{3.8317\norm{x}}\right)^2,\ K(0)=1,
\end{equation}
where $J_1$ is the first-order Bessel function of the first kind. The constant factor $3.8317$ scales the kernel so that its minimum occurs at approximately $\norm{x}=1$.

\begin{figure}[!t]
\centering
\includegraphics[scale=0.5]{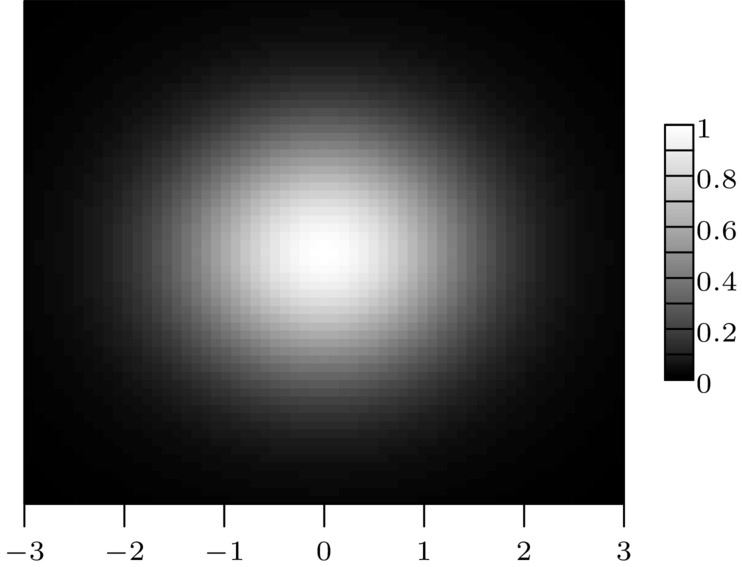}
\includegraphics[scale=0.5]{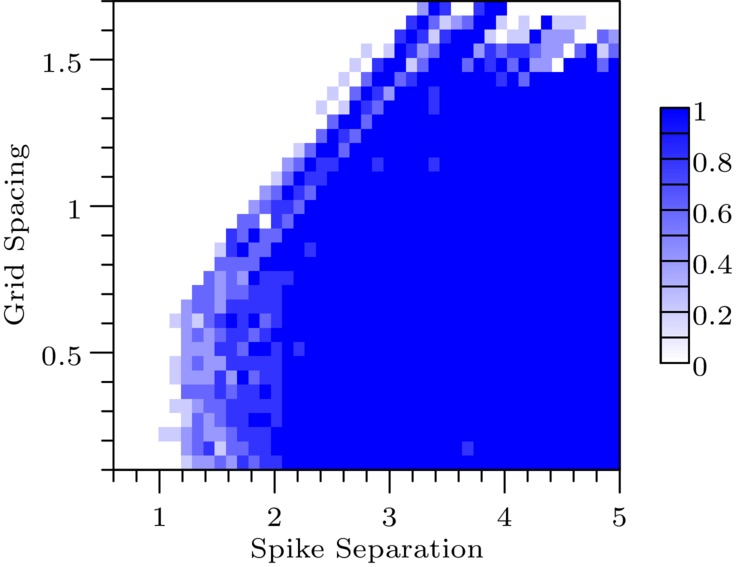}
\caption[Microscopy Simulation]{\label{fig:ExactCVXMicro}Microscopy kernel used in our simulations (left). Fraction of exact recovery for different values of the grid spacing and spike separation (right).}
\end{figure}
\begin{figure}[!t]
\centering
\includegraphics[scale=0.5]{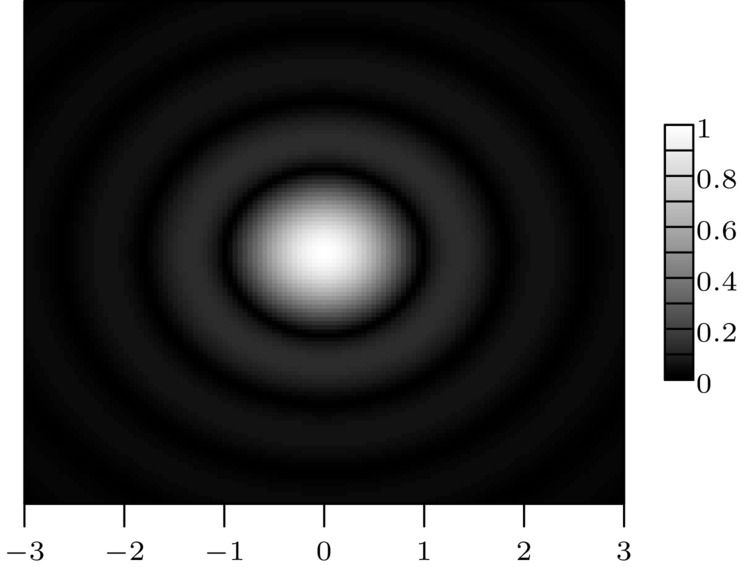}
\includegraphics[scale=0.5]{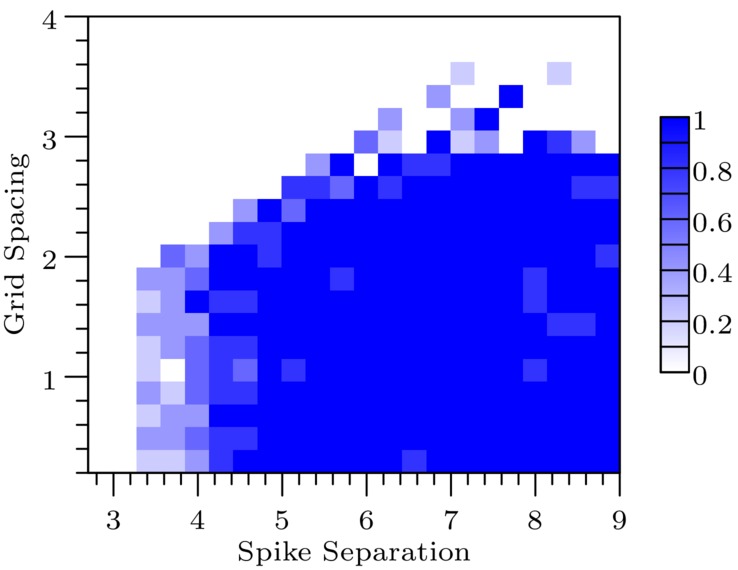}
\caption[Telescopy Simulation]{\label{fig:ExactCVXTele} Square root of the Airy kernel defined in Eq.\eqref{eq:TeleKernel} (left) (we take the square root to make the surrounding rings more apparent). Fraction of exact recovery for different values of the grid spacing and spike separation (right).}
\end{figure}

\Cref{fig:ExactCVXMicro} shows the results of repeating the numerical experiments described in \Cref{sec:ExactSimulation} for the microscopy kernel. 
The parameters $\spikesep$ and $\gridsep$ are scaled by $\sigma_0$ where $2\sigma_0^2 = 1.72^2/2$, the exponential denominator in the first term in the kernel's sum.
The results are similar to those for the Gaussian kernel and suggest that successful recovery can be achieved when the spike separation $\spikesep\geq 2\sigma_0$. 

The Airy pattern has a significantly different shape from the microscopy and Gaussian kernels. \Cref{fig:ExactCVXTele} shows an image of this kernel \eqref{eq:TeleKernel}; note the faint sequence of rings surrounding the bright center. Nonetheless, the results of repeating the numerical results described in \Cref{sec:ExactSimulation} for this kernel are very similar, as shown in \Cref{fig:ExactCVXTele}; exact recovery again occurs for a large enough spike separation.

\section{Conclusion}\label{sec:Conclusion}
In this work we prove a sampling theorem for Gaussian deconvolution in two-dimensions.  In particular, we give an explicit region of minimum-separation and grid-spacing values where convex optimization exactly recovers the true signal. This extends the results of \cite{bernstein2018} to two dimensions, a setting important in many applications.  We give numerical evidence that our results extend to two non-Gaussian convolution kernels arising in microscopy and telescopy.
\bdb{
The extension to two dimensions is accomplished by a geometric approach, where the plane is partitioned into sets that contain at most one spike.
This provides a strategy to extend the proof to higher dimensions: define an interpolation function using nearby samples (this would require $d+1$ samples per spike in $d$ dimensions), reparametrize the function using extensions of bumps and waves to $d$ dimensions, partition the space, and analyze the corresponding function on the partition by exploiting a minimum separation condition.}

An interesting direction for future research is to find a new dual-certificate construction that utilizes all of the sample data (our presented construction only uses the closest three samples to each spike).  This could bridge the gap between our theoretical results and the numerical experiments in \Cref{sec:ExactSimulation}. Other directions of future research include the analysis of the discretization error incurred when solving the $\ell_1$-norm minimization problem on a grid, and obtaining recovery guarantees for blind-deconvolution settings where the convolution kernel must be jointly estimated from the data.

\subsection*{Acknowledgements}
This research was enabled by NSF NRT-HDR Award 1922658. C.F. was supported by NSF award DMS-1616340. B.B.~is generously supported by the MacCracken Fellowship, and the Isaac Barkey and Ernesto Yhap Fellowship.

\bibliographystyle{abbrv}
\bibliography{refs} 

\appendix
\renewcommand{\appendixpagename}{Appendix}
\appendixpage
\section{Proof of \Cref{cor:DiscreteTV} and \Cref{prop:DualCert}}\label{sec:DualCertProof}
\DiscreteTV*
\begin{proof}
Problem \eqref{pr:TVnormDisc} is equivalent to problem \eqref{pr:TVnorm} restricted to measures supported on $G$. As
a result, any respective solutions $\hat{a}$ and $\hat{\mu}$ must satisfy $\normone{\hat{a}} \geq \normTV{\hat{\mu}}$. By \Cref{thm:Exact}, \eqref{pr:TVnorm} is
uniquely minimized by $\mu$. Since $\mu$ is supported on $T$, $a$ is the unique solution of \eqref{pr:TVnormDisc}.
\end{proof}

\DualCert*

\begin{proof}
This proof is identical to that in \cite[Appendix A]{bernstein2018} and is included for completeness. Suppose $\nu$ is feasible for problem \eqref{pr:TVnorm}. Then if $h=\nu-\mu$ by the Lebesgue decomposition $h=h_T+h_{T^c}$ where $h_T$ is absolutely continuous with respect to $\mu$ and $h_{T^c}$ is singular with respect to $\mu$. Thus $h_T$ can be written as
\begin{equation}
h_T = \sum_{j=1}^N b_j\delta_{t_j}
\end{equation}
where $t_j\in T$. Since both $\mu$ and $\nu$ are feasible we have that $(K\ast h)(s_i)=0$ for $s_i\in S$. From the assumptions, let $Q(t)=\sum_{i=1}^n q_i K(s_i-t)$ be such that $Q(t_j)=\sign(b_j)$. Then
\begin{align}
0 &= \sum_{i=1}^n q_i (K\ast h)(s_i) = \sum_{i=1}^n q_i \int K(s_i-t)\ dh(t)\\
& = \int Q(t)\ dh(t)
=\normTV{h_T} + \int Q(t)\ dh_{T^c}(t).
\end{align}
Then the total variation norm of $\nu$ is
\begin{align}
\normTV{\nu}&=\normTV{\mu+h_{T}} + \normTV{h_{T^c}}\\
&\geq \normTV{\mu} - \normTV{h_{T}} + \normTV{h_{T^c}}\\
&=\normTV{\mu} + \int Q(t)\ dh_{T^c}(t) + \normTV{h_{T^c}}\\
&\geq \normTV{\mu}
\end{align}
The last inequality is strict if $\normTV{h_{T^c}}>0$ since $\abs{Q(t)}<1$ on $T^c$. Thus $\mu$ is optimal. Since any other optimal solution must be supported on $T$ we have $\normTV{h_T}=\int Q(t)\ dh_{T^c}(t)=0$, and so $\mu$ is unique.
\end{proof}

\section{Reference Table for Exact Recovery}\label{sec:RecoveryTable}
\bdb{\Cref{tab:RecoveryTable} provides a quantitative description of the boundaries of the recovery region for \Cref{thm:Exact} shown in \Cref{fig:RecoveryMapFirst}, where the values for $\gridsep$ and $\spikesep(T)$ are given in units of $\sigma$. The right column indicates the smallest $\spikesep(T)$ at which \Cref{thm:Exact} guarantees recovery if $\gridsep$ belongs to the intervals in the left column.}
\begin{table}[h]
\centering
\begin{tabular}{|c|c|}
    \hline
    $\gridsep$ (in $\sigma$) & $\spikesep(T)$ (in $\sigma$) \\ \hline
    0.10--0.15 & 4.10\\
    0.15--0.20 & 4.10\\
    0.20--0.25 & 4.10\\
    0.25--0.30 & 4.10\\
    0.30--0.35 & 4.10\\
    0.35--0.40 & 4.10\\
    0.40--0.45 & 4.10\\
    0.45--0.50 & 4.15\\
    \hline
\end{tabular}\qquad
\begin{tabular}{|c|c|}
    \hline
    $\gridsep$ (in $\sigma$) & $\spikesep(T)$ (in $\sigma$) \\ \hline
    0.50--0.55 & 4.15\\
    0.55--0.60 & 4.20\\
    0.60--0.65 & 4.25\\
    0.65--0.70 & 4.30\\
    0.70--0.75 & 4.35\\
    0.75--0.80 & 4.40\\
    0.80--0.85 & 4.50\\
    0.85--0.89 & 4.55\\
    \hline
\end{tabular}
\caption{\label{tab:RecoveryTable} The table provides the minimum value of $\spikesep(T)$ at which recovery is guaranteed for each interval for $\gridsep$. The values for $\gridsep$ and $\spikesep(T)$ are given in units of $\sigma$.}
\end{table}

\section{Proofs for Bumps and Waves}
\subsection{Proof that Bump and Wave Functions Exist} \label{sec:BWExistProof}
Here we show that the bump and wave functions used in our dual certificate construction exist.  We also establish that the bumps are formed from a \textit{non-negative} linear combination of sample-centered Gaussians, a fact used in \Cref{sec:EVEnvelopes} where we bound the Hessian of the bump.

Combining equations \eqref{eq:BWdefs}, \eqref{eq:Binterp} and \eqref{eq:Winterp} for a fixed $j$ gives the linear system (where $K$ is the Gaussian kernel):
\begin{equation}\label{eq:bwlin}
\begin{bmatrix}
K(s_j^1-t_j)&K(s_j^2-t_j)&K(s_j^3-t_j)\\
\partial_xK(s_j^1-t_j)&\partial_xK(s_j^2-t_j)&\partial_xK(s_j^3-t_j)\\
\partial_yK(s_j^1-t_j)&\partial_yK(s_j^2-t_j)&\partial_yK(s_j^3-t_j)
\end{bmatrix}
\begin{bmatrix}
\bcone_j&\wocone_j&\wtcone_j\\
\bctwo_j&\woctwo_j&\wtctwo_j\\
\bcthree_j&\wocthree_j&\wtcthree_j
\end{bmatrix}
=
\begin{bmatrix}
1&0&0\\
0&1&0\\
0&0&1
\end{bmatrix}
\end{equation}
Thus the coefficient matrix is the inverse of the matrix with kernel measurements at the spike $t_j$. We will show that this inverse exists and thus these functions are well-defined.
Since we are considering the bump and waves from a single spike in the following proof we omit the subscript $j$ where it is convenient. Additionally the first and second coordinates of $t$ and $s^i$ are denoted by $t_{(1)}$ and $t_{(2)}$  and $s^i_{(1)}$ and $s^i_{(2)}$.
\begin{restatable}[Proof in \Cref{sec:BWExistProof}]{lemma}{BWExist}\label{lem:BWExist}
The coefficients for the bump $B_j$ and waves $W^1_j$ and $W^2_j$ are given by
\begin{align}
\label{eq:coeffmat}
\begin{bmatrix}
\bcone_j&\wocone_j&\wtcone_j\\
\bctwo_j&\woctwo_j&\wtctwo_j\\
\bcthree_j&\wocthree_j&\wtcthree_j
\end{bmatrix}
&=\frac{1}{D}
\begin{bmatrix}
e^{\|s^1-t\|^2/2}&0&0\\
0&e^{\|s^2-t\|^2/2}&0\\
0&0&e^{\|s^3-t\|^2/2}
\end{bmatrix}
\begin{bmatrix}
D_1 & (s^2_{(2)}-s^3_{(2)}) & (s^3_{(1)}-s^2_{(1)})\\
D_2 & (s^3_{(2)}-s^1_{(2)}) & (s^1_{(1)}-s^3_{(1)})\\
D_3 & (s^1_{(2)}-s^2_{(2)}) & (s^2_{(1)}-s^1_{(1)})
\end{bmatrix}
\end{align}
when $D\neq0$.
Here
\begin{equation}
\label{eq:deq}
\begin{aligned}
D_1&=(s^2_{(1)}-t_{(1)})(s^3_{(2)}-t_{(2)})-(s^3_{(1)}-t_{(1)})(s^2_{(2)}-t_{(2)})\\
D_2&=(s^3_{(1)}-t_{(1)})(s^1_{(2)}-t_{(2)})-(s^1_{(1)}-t_{(1)})(s^3_{(2)}-t_{(2)})\\
D_3&=(s^1_{(1)}-t_{(1)})(s^2_{(2)}-t_{(2)})-(s^2_{(1)}-t_{(1)})(s^1_{(2)}-t_{(2)})
\end{aligned}
\end{equation}
with $D=D_1+D_2+D_3$.
\end{restatable}
\begin{proof}
If
\begin{align}
& \begin{bmatrix}
 K(s^1-t)&K(s^2-t)&K(s^3-t)\\
 \partial_xK(s^1-t)&\partial_xK(s^2-t)&\partial_xK(s^3-t)\\
 \partial_yK(s^1-t)&\partial_yK(s^2-t)&\partial_yK(s^3-t)
 \end{bmatrix}\\
 &=
\begin{bmatrix}
K(s^1-t) & K(s^2-t) & K(s^3-t)\\
-(t_{(1)}-s_{(1)}^1)K(s^1-t) & -(t_{(1)}-s_{(1)}^2)K(s^2-t) & -(t_{(1)}-s_{(1)}^3)K(s^3-t)\\
-(t_{(2)}-s_{(2)}^1)K(s^1-t) & -(t_{(2)}-s_{(2)}^2)K(s^2-t) & -(t_{(2)}-s_{(2)}^3)K(s^3-t)
\end{bmatrix}\\
&=
\begin{bmatrix}
1 & 1 & 1\\
(s_{(1)}^1-t_{(1)}) & (s_{(1)}^2-t_{(1)}) & (s_{(1)}^3-t_{(1)})\\
(s_{(2)}^1-t_{(2)}) & (s_{(2)}^2-t_{(2)}) & (s_{(2)}^3-t_{(2)})
\end{bmatrix}
\begin{bmatrix}
e^{-\|s^1-t\|^2/2}&0&0\\
0&e^{-\|s^2-t\|^2/2}&0\\
0&0&e^{-\|s^3-t\|^2/2}
\end{bmatrix}
\end{align}
then the result follows by inverting.
\end{proof}
\begin{lemma}\label{lem:crossprod}
Let $D$ be defined as in \Cref{lem:BWExist}.
Then $|D|=\gridsep^2\neq0$ and the $D_i$'s each have the same sign, so $0\leq D_i/D\leq1$ for $i=1,2,3$.  \end{lemma}
\begin{proof}
By assumption the spike location $t$ sits in the right triangle with vertices given by $s^1$, $s^2$, and $s^3$, its three closest samples.  Let $v^i:=s^i-t$ for $i=1,2,3$.  Then the determinants of the matrices $[v^2\ v^3]$, $[v^3\ v^1]$, and $[v^1\ v^2]$ all have the same sign, since they are the signed areas of parallelograms with the same orientation.  These determinants are precisely $D_1$, $D_2$, and $D_3$, respectively. Furthermore, $|D|=|D_1+D_2+D_3|$ gives twice the area of the right triangle with vertices $s^1$, $s^2$, and $s^3$, which is exactly $\gridsep^2$.
\end{proof}

\subsection{Envelope Construction}\label{sec:Envelopes}

\Cref{eq:Envelopes} presents a construction of the envelopes providing radially symmetric upper bounds for any bump or wave and their derivatives. These envelopes must account for all possible positions of the spike relative to its three nearest sample points. To compute such an envelope we rely on Mathematica's Interval Arithmetic package \cite{Mathematica}. A detailed description of interval arithmetic is given in Appendix B.6 of \cite{bernstein2018}. This package computes hard limits on the possible range of a specified function depending on the range of its arguments and the operators used. If the range of each argument is narrow then bounds on the range of values the function takes is also narrow and a relatively sharp upper bound on the function can be obtained for that region in parameter space.

Below we describe how we compute the envelope for a bump function, but the same method applies to the wave, and the bump and wave derivatives.
Five parameters determine the value of a bump $B_1$ \bdb{corresponding to a spike location $t_1$} at a given position $t\in\RR^2$: the grid spacing $\gridsep$, the two-dimensional spike offset $u=t_1-s_1^1$, and the two-dimensional positional argument $t$.
\bdb{For simplicity of exposition, we assume without loss of generality that $t_1$ is located at the origin.}
The offset $u$ gives the difference between the spike and its nearest sample point.  Combined with the grid separation $\gridsep$, the offset determines the position of the other two samples $s_1^2$ and $s_1^3$. $\gridsep$ and $u$ affect the values of the coefficients $\coeffone$, $\coefftwo$ and $\coeffthree$ and thus the shape of the function while $t$ specifies where the function is being measured. We partition the space for these parameters into sections $I_\gridsep$, $I_u$ and $I_t$ defined as follows:
\begin{equation}
\begin{gathered}
I_\gridsep(k)=[0.1+0.8(k-1)/16,0.1+0.8k/16],\ 1\leq k\leq 16,\\
I_t(j,k)=[-10+(j-1)/40,-10+j/40]\times[-10+(k-1)/40,-10+k/40],\ 1\leq j,k\leq 800,\\
I_u(j,k)=[(j-1)\gridsep/40, j\gridsep/40]\times [(k-1)\gridsep/40, k\gridsep/40],\ 1\leq j,k\leq 20.
\end{gathered}
\end{equation}
\bdb{The number of intervals in each partition are selected so that the resulting envelopes are sharp enough to facilitate the remainder of the proof.}
Fixing values for $k_i$ for $i=1,2,3,4,5$ we use interval arithmetic to compute an upper bound $\tilde{B}$ satisfying
\begin{equation}
\wt{B}(k_1,k_2,k_3,k_4,k_5)
\geq \sup_{\substack{\gridsep\in I_\gridsep(k_1)\\
t\in I_t(k_2,k_3)\\ t_1-s_1^1\in I_u(k_4,k_5)}} |B_1(t;s_1^1,s_1^2,s_1^3)|.
\end{equation}
By symmetry, and since we are taking absolute values, we only consider values of $u$ with non-negative coordinates that are smaller than $\gridsep/2$.  We only consider values of $t$ in $[-10,10]^2$ since by \Cref{lem:BumpBounds} and \Cref{lem:WaveBounds} all bumps, waves and their derivatives are smaller in absolute value than $2\cdot 10^{-9}$ for $\norm{t}\geq 10$.

Using $\wt{B}$ we can compute an upper bound on $\enva{B}(r)$ for $r\leq 10$ and for a fixed value of $\gridsep$ as follows:
\begin{equation}
    \enva{B}(r)
    \leq \max_{k_2,k_3,k_4,k_5}\max(\wt{B}(k_1,k_2,k_3,k_4,k_5),2\cdot 10^{-9}),
\end{equation}
where $\gridsep\in I_\gridsep(k_1)$, $k_2,k_3$ range over all values where $I_t(k_2,k_3)$ contains a point $t$ with $\|t\|\geq r$, and $k_4,k_5$ take all possible values.
Since there are only finitely many possible intervals $I_t(j,k)$, there are only finitely many possible values of $\enva{B}(r)$.
Note that a separate envelope $\enva{B}(r)$ is computed for each $k_1$ specifying the range of $\gridsep$.
Wave and derivative envelopes are calculated similarly.

\subsubsection{Bump Directional Derivative Envelope Construction}\label{sec:ConstructingDDEnvelopes}

To obtain a discretized representation of the upper bound $\Dt$ in \eqref{eq:RawEnvelopes} for the directional derivative of bump functions, define the following function:
\begin{equation}\label{eq:Dformula}
\Dt(r;\gridsep)=\sup_{\substack{\norm{t-t_1}=r\\ \text{$s^1_1,s^2_1,s^3_1$ nearest $t_1$}}}
\partial_x B_1(t)\cdot\frac{(t-t_1)_{(1)}}{\norm{t-t_1}}+
\partial_y B_1(t)\cdot\frac{(t-t_1)_{(2)}}{\norm{t-t_1}}\ ,
\end{equation}
where the supremum is taken over all configurations of $t_1$'s nearest three samples $s^1_1$, $s^2_1$, and $s^3_1$ with grid spacing $\gridsep$ and points $t=(t_{(1)},t_{(2)})$ a distance $r$ from $t_1$.
For convenience we assume that $t_1$ sits at the origin.
Note $\Dt$ is a function of $\gridsep$, is not monotonic as a function of $\|t\|$, and should be negative when $\norm{t}$ is small for sufficiently large $\spikesep$ (i.e., the bump envelope decays). 

By partitioning the parameters for $\gridsep$, spike offset $u$ and positional argument $t$ into intervals $I_\gridsep$, $I_u$ and $I_t$ we can use Interval Arithmetic to compute a non-monotonic upper bound on $\Dt$ that bounds the directional derivative.
Recall that $I_t$ and $I_u$ are intervals in $\RR^2$ or rectangles.
Fix $k_1$ to specify $\gridsep$'s range, i.e. $\gridsep\in I_\gridsep(k_1)$. Define
\begin{gather}
\wt{\Dt}(k_1,k_2,k_3,k_4,k_5):=
\sup_{\substack{\gridsep\in I_\gridsep(k_1)\\
t\in I_t(k_2,k_3)\\ s_j^1\in I_u(k_4,k_5)}}
\partial_x B(t;t_1,s_j^1,s_j^2,s_j^3)\cdot\frac{t_{(1)}}{\norm{t}}+
\partial_y B(t;t_1,s_j^1,s_j^2,s_j^3)\cdot\frac{t_{(2)}}{\norm{t}}.
\end{gather}
Then
\begin{align}
\Dt(r;\gridsep)&\leq\max_{k_2,k_3,k_4,k_5}\max(\wt{\Dt}(k_1,k_2,k_3,k_4,k_5),2\cdot 10^{-9})\label{eq:DQDisc}
\end{align}
where the maximum is taken over values of $k_2$, $k_3$ where $I_t(k_2,k_3)$ contains a point $t$ with $\|t\|=r$, and $k_4$, $k_5$ take all possible values.
The right side of \eqref{eq:DQDisc} gives a discretized envelope, and a separate envelope is obtained for each choice of $k_1$.

\subsubsection{Eigenvalue Envelope Construction}\label{sec:EVEnvelopes}

To bound the largest eigenvalue of the bumps, and the largest absolute eigenvalues of the bump and waves, we construct envelope functions for $\lambda(B)$, $\lambda(W^i)$ and $\lambda(B)_\infty$ just as we did for the functions and their derivatives.
As in \eqref{eq:Envelopes} and \eqref{eq:RawEnvelopes} define $\lambda(B)_\infty$ by taking the supremum over points at the same distance from $t_j$ and over all positions of $t_j$ with respect to its three nearest samples $s^1_j$, $s^2_j$ and $s^3_j$ and unit vectors $v$:
\begin{equation}\label{eq:EVnonmon}
\lambda(B)_\infty(r):=\sup_{\substack{\norm{t-t_j}=r\\ s_j^1,s_j^2,s_j^3\mbox{ nearest }t_j\\ \norm{v}=1}}v^T\nabla^2 B_j(t;t_j,s^1_j,s^2_j,s^3_j) v.
\end{equation}
$\ltamaxb$ and $\enva{\lambda(W^i)}$ are monotonized by taking the supremum over $\norm{t-t_1}\geq r$ for $i\in{1,2}$:
\begin{equation}\label{eq:EVmon}
\begin{aligned}
\ltamaxb(r)&:=\sup_{\substack{\norm{t-t_j}\geq r\\ s_j^1,s_j^2,s_j^3\mbox{ nearest }t_j\\ \norm{v}=1}}\abs{v^T\nabla^2 B_j(t;t_j,s^1_j,s^2_j,s^3_j) v}\\
\enva{\lambda(W^i)}(r)&:=\sup_{\substack{\norm{t-t_j}\geq r\\ s_j^1,s_j^2,s_j^3\mbox{ nearest }t_j\\ \norm{v}=1}}\abs{v^T\nabla^2 W^i_j(t;t_j,s^1_j,s^2_j,s^3_j) v}.
\end{aligned}
\end{equation}
To simplify these for something easier to compute, we first derive a form for the contribution from each of the three Gaussian terms in a bump or wave. The Hessian of $f(x) = e^{-\norm{x}^2/2}$ is:
\begin{equation}
\nabla^2 f(x)=\begin{bmatrix}
(x_1^2-1)&x_1x_2\\
x_1x_2&(x_2^2-1)
\end{bmatrix}\cdot e^{-\norm{x}^2/2},
\end{equation}
so the eigenvalues are $\lambda e^{-\norm{x}^2/2}$ such that
\begin{equation}
\begin{aligned}
0&=(x_1^2-1-\lambda)(x_2^2-1-\lambda)-x_1^2x_2^2\\
&=\lambda^2+\lambda(2-\norm{x}^2)+(1-\norm{x}^2).
\end{aligned}
\end{equation}
Consequently,
\begin{equation}
\lambda=\norm{x}^2-1,-1.
\end{equation}
If $g(x)=\kappa f(x)$ then the largest eigenvalue of $\nabla^2 g$ is
\begin{equation}\label{eq:eigenvalue}
\lambda(g)(x)=\max(\kappa(\norm{x}^2-1), -\kappa)e^{-\norm{x}^2/2}.
\end{equation}
Since every bump and wave is a sum of three weighted Gaussians \eqref{eq:BWdefs}, their largest eigenvalue at any point is less than the sum of the largest eigenvalues from each Gaussian. For a bump $B_1$ where $Q(t_1)=1$, each bump coefficient $\bcone_1,\bctwo_1,\bcthree_1\geq 0$, so
\begin{equation}
\begin{aligned}
v^T \nabla^2 B_1(t) v&\leq
\bcone_1(\norm{s^1_1-t}^2-1)e^{-\norm{s^1_1-t}^2/2}+\bctwo_1(\norm{s^2_1-t}^2-1)e^{-\norm{s^2_1-t}^2/2}\\
&\qquad+\bcthree_1(\norm{s^3_1-t}^2-1)e^{-\norm{s^3_1-t}^2/2}.
\end{aligned}
\end{equation}
The largest absolute eigenvalue of the bump is bounded as follows:
\begin{equation}
\begin{aligned}\label{eq:BumpEVAbsSum}
\abs{v^T\nabla^2 B_j(t)v}&\leq\bcone_j\max(\norm{s^1_j-t}^2-1,1)e^{-\norm{s^1_j-t}^2/2}\\
&\qquad+\bctwo_j\max(\norm{s^2_j-t}^2-1,1)e^{-\norm{s^2_j-t}^2/2}\\
&\qquad+\bcthree_j\max(\norm{s^3_j-t}^2-1,1)e^{-\norm{s^3_j-t}^2/2}.
\end{aligned}
\end{equation}
A similar bound holds for both waves, where we must now account for each coefficient's sign:
\begin{equation}
\begin{aligned}
\abs{v^T\nabla^2W^1_j(t)v}&\leq
\abs{\wocone_j}\max(\norm{s^1_j-t}^2-1,1)e^{-\norm{s^1_j-t}^2/2}\\
&\qquad+\abs{\woctwo_j}\max(\norm{s^2_j-t}^2-1,1)e^{-\norm{s^2_j-t}^2/2}\\
&\qquad+\abs{\wocthree_j}\max(\norm{s^3_j-t}^2-1,1)e^{-\norm{s^3_j-t}^2/2}.
\end{aligned}\label{eq:WaveEVAbsSum}
\end{equation}
By \eqref{eq:coeffmat} one of $\woctwo_j$ and $\wocthree_j$ will be zero and similarly for $\wtctwo_j$ and $\wtcthree_j$.


We discretize the upper bounds in \eqref{eq:EVnonmon} and \eqref{eq:EVmon} using the same methods described for the bump, wave and derivative envelopes. As before we partition the parameters for $\gridsep$, positional argument $t$ and spike offset $u$ into intervals $I_\gridsep$, $I_t$ and $I_u$. Recall $I_t$ and $I_u$ both are intervals in $\RR^2$ or rectangles. For the bump we use the previous choices of intervals, but since the waves' coefficient signs will affect the largest wave eigenvalues we extend the range of $I_u$ to ensure all sign combinations for the coefficients of $W^1$ and $W^2$ are considered:
\begin{equation}
I_u(j,k)=[(j-1)\gridsep/40, j\gridsep/40]\times [(k-1)\gridsep/40, k\gridsep/40],\ -19\leq j,k\leq 20, 
\end{equation}
We compute a discretized upper bound on the largest eigenvalue of bumps with parameters in particular intervals using Interval Arithmetic:
\begin{align}
\wt{\lambda(B)}_\infty(k_1,k_2,k_3,k_4,k_5)&:=
\begin{array}{c}
\sup\\
{\substack{\gridsep\in I_\gridsep(k_1)\\
t\in I_t(k_2,k_3)\\ s_j^1\in I_u(k_4,k_5)\\}}
\end{array}
\begin{array}{l}
\bcone_1(\norm{s^1_1-t}^2-1)e^{-\norm{s^1_1-t}^2/2}\\
\quad+\bctwo_1(\norm{s^2_1-t}^2-1)e^{-\norm{s^2_1-t}^2/2}\\
\quad+\bcthree_1(\norm{s^3_1-t}^2-1)e^{-\norm{s^3_1-t}^2/2}
\end{array}
\\
&\geq\sup_{\substack{\gridsep\in I_\gridsep(k_1)\\
t\in I_t(k_2,k_3)\\ s_j^1\in I_u(k_4,k_5)}}
v^T \nabla^2 B_1(t;0,s_j^1,s_j^2,s_j^3) v
\end{align}
Then a discretized envelope for $\lambda(B)_\infty(r)$ for $r\leq 10$ is obtained for all bumps with a fixed value of $\gridsep\in I_\gridsep(k_1)$ as follows:
\begin{equation}\label{eq:BumpEVMax}
\lambda(B)_\infty(r)
\leq\max_{k_2,k_3,k_4,k_5}\max(\wt{\lambda(B)}_\infty(k_1,k_2,k_3,k_4,k_5),2\cdot 10^{-9}),
\end{equation}
where $k_2$, $k_3$ range over all values where $I_t(k_2,k_3)$ contains a point $t$ with $\|t\|=r$, and $k_4$, $k_5$ take all possible values for all spike offsets.

We use the same method for a monotonic bound on the largest absolute eigenvalue of bumps and waves. A discretized bound for $\ltamaxb$ is computed using Interval Arithmetic from
\begin{align}
\wt{\enva{\lambda(B)}}(k_1,k_2,k_3,k_4,k_5)&:=
\begin{array}{c}
\sup\\
{\substack{\gridsep\in I_\gridsep(k_1)\\
t\in I_t(k_2,k_3)\\ s_j^1\in I_u(k_4,k_5)}}
\end{array}
\begin{array}{l}
\abs{\bcone_j}\max(\norm{s^1_j-t}^2-1,1)e^{-\norm{s^1_j-t}^2/2}\\
\qquad+\abs{\bctwo_j}\max(\norm{s^2_j-t}^2-1,1)e^{-\norm{s^2_j-t}^2/2}\\
\qquad+\abs{\bcthree_j}\max(\norm{s^3_j-t}^2-1,1)e^{-\norm{s^3_j-t}^2/2}
\end{array}
\\
&\geq \sup_{\substack{\gridsep\in I_\gridsep(k_1)\\
t\in I_t(k_2,k_3)\\ s_j^1\in I_u(k_4,k_5)}}
\abs{v^T \nabla^2 B_1(t;0,s_j^1,s_j^2,s_j^3) v}.
\end{align}
Then an envelope for all bumps is obtained by taking the maximum over interval choices:
\begin{equation}
\ltamaxb(r)
\leq\max_{k_2,k_3,k_4,k_5}\max(\wt{\enva{\lambda(B)}}(k_1,k_2,k_3,k_4,k_5),2\cdot 10^{-9}).
\end{equation}
Different from \eqref{eq:BumpEVMax}, here $k_2$ and $k_3$ range over values where $I_t(k_2,k_3)$ contains a point $t$ such that $\norm{t}\geq r$, monotonizing the envelope, and $k_4,\ k_5$ take all possible values. The same is done for wave envelopes $\ltmaxwo$ and $\ltmaxwt$ using the extended range of $I_u$ for $k_4$ and $k_5$.

\subsection{Proof of \Cref{lem:SchurInvEasy}}\label{sec:SchurInvProof}
Recall that we rewrite the linear system \eqref{eq:Interpolation} 
using the bump and wave parametrization as
\begin{equation}
\begin{bmatrix}
\BB & \WWo & \WWt\\
\DoB & \DoWo& \DoWt\\
\DtB & \DtWo& \DtWt
\end{bmatrix}
\begin{bmatrix}
\bcoeff\\
\wocoeff\\
\wtcoeff
\end{bmatrix}
=
\begin{bmatrix}
\signs\\
0\\
0
\end{bmatrix}
\label{eq:BlockEquationSupp}
\end{equation}
for some vectors $\bcoeff$, $\wocoeff$, $\wtcoeff$.
For clarity we rewrite this as
\begin{equation}
\label{eq:clarity}
\begin{bmatrix}
\BB&\WW\\
\BBp&\WWp
\end{bmatrix}
\begin{bmatrix}
\bcoeff\\
\Gamma
\end{bmatrix}
=
\begin{bmatrix}
\signs\\
0
\end{bmatrix},
\end{equation}
where
\begin{equation}
\BBp:=\begin{bmatrix}
\DoB\\
\DtB
\end{bmatrix},\
\WW:=\begin{bmatrix}
\WWo & \WWt
\end{bmatrix},\
\WWp:=\begin{bmatrix}
\DoWo& \DoWt\\
\DtWo& \DtWt
\end{bmatrix},\mbox{ and }
\Gamma:=\begin{bmatrix}
\wocoeff\\
\wtcoeff
\end{bmatrix}.
\end{equation}
Denote the matrix on the left side of \eqref{eq:clarity} as $\MM$.
The invertibility of $\MM$ in \eqref{eq:clarity} implies the existence of $\bcoeff$, $\wocoeff$, $\wtcoeff$ satisfying \Cref{lem:Invertibility}. The next two lemmas relate the norms of these matrices to the invertibility of $M$ along with useful bounds on the associated coefficients $\bcoeff$, $\wocoeff$, and $\wtcoeff$.
The core idea is that the diagonal elements will be exactly one by construction, and the off-diagonal elements of $\MM$ will be close to zero when the spikes are sufficiently separated.
The lemmas use the following matrices: 
\begin{align}
\SSo&:=\DoWo-\DoWt\DtWti\DtWo\\
\SSt&:=\DoB-\DoWt\DtWti\DtB\\
\SSth&:=\BB-\WWo\SSoi\SSt+\WWt\DtWti(\DtWo\SSoi\SSt-\DtB)
\end{align}

\begin{restatable}{lemma}{SchurInv}\label{lem:SchurInv}
Suppose
\begin{enumerate}
\item $\norminf{\II-\DtWt}<1$,
\item $\norminf{\II-\DoWo}+\norminf{\DoWt}\norminf{\DtWti}\norminf{\DtWo}<1$, and
\item $\norminf{\II-\BB}+\norminf{\WW}\norminf{\WWpi}\norminf{\BBp}<1$.
\end{enumerate}
Then $\SSoi$ and $\SSth^{-1}$ exist. 
\end{restatable}

\begin{proof}
For any matrix $A\in\RR^{n\times n}$ such that $\norminf{ A} < 1$ the Neumann
series $\sum_{j=0}^\infty A^j$ converges to $(\II-A)^{-1}$.  By the
triangle inequality and the submultiplicativity of the $\infty$-norm, this gives
\begin{equation}\label{eq:neumann}
\norminf{(\II-A)^{-1}} \leq \sum_{j=0}^\infty \norminf{A}^j =
\frac{1}{1-\norminf{A}}.
\end{equation}
Setting $A=\II-\DtWt$ proves $\DtWt$ is invertible. Observe that $\SSo=\DoWo-\DoWt\DtWti\DtWo$ is the Schur complement of $\DtWt$, and by the triangle inequality and the second assumption, \begin{equation}
\norminf{I-\SSo}\leq\norminf{I-\DoWo}+\norminf{\DoWt}\norminf{\DtWti}\norminf{\DtWo}<1.
\end{equation}
Thus $\SSo$ is invertible and consequently so is $\WWp$. Then the Schur complement of $\WWp$ is
\begin{align}
\begin{split}
\BB-\WW\WWpi\BBp
&=\BB-\WWo\SSoi\DoB+\WWt\DtWti\DtWo\SSoi\DoB+\WWo\SSoi\DoWt\DtWti\DtB\\
&\qquad-\WWt\DtWti\DtB-\WWt\DtWti\DtWo\SSoi\DoWt\DtWti\DtB
\end{split}\\
\begin{split}
&=\BB-\WWo\SSoi(\DoB-\DoWt\DtWti\DtB)\\
&\qquad+\WWt\DtWti(-\DtB+\DtWo\SSoi\DoB-\DtWo\SSoi\DoWt\DtWti\DtB)
\end{split}\\
\begin{split}
&=\BB-\WWo\SSoi\SSt+\WWt\DtWti(\DtWo\SSoi\SSt-\DtB)
\end{split}\\
\begin{split}
&=\SSth
\end{split}
\end{align}
so from the last assumption
\begin{equation}
\norminf{I-\SSth}\leq\norminf{I-\BB}+\norminf{\WW}\norminf{\WWpi}\norminf{\BBp}<1.
\end{equation}
Thus $\SSth$ is invertible.
\end{proof}

\begin{restatable}{lemma}{Schur}\label{lem:Schur}
Suppose $\DtWti$, $\SSoi$ and $\SSthi$ all exist. Then
\begin{enumerate}
\item $\WWp$ is invertible,
\item $\MM$ is invertible,
\item $\norminf{\bcoeff}\leq\norminf{\SSthi},$
\item $\norminf{\wocoeff}, \norminf{\wtcoeff}\leq\norminf{\SSoi}\norminf{\SSt}\norminf{\SSthi}$,
\item $\abs{\alpha_i-\signs_i}\leq\norminf{\SSthi}\norminf{I-\SSth}$ for all $i$,
\end{enumerate}
where
\begin{equation}\label{eq:matrixformulafirst}
\norminf{\SSoi}\leq(1-\norminf{\II-\DoWo}-\norminf{\DoWt}\norminf{\DtWti}\norminf{\DtWo})^{-1},\\
\end{equation}
\begin{equation}\label{eq:matrixformula2}
\norminf{\SSt}\leq\norminf{\DoB}+\norminf{\DoWt}\norminf{\DtWti}\norminf{\DtB},\\
\end{equation}
\begin{equation}\label{eq:matrixformula3}
\norminf{\SSthi}\leq(1-\norminf{\II-\SSth})^{-1},
\end{equation}
\begin{equation}\label{eq:matrixformula4}
\begin{split}
\norminf{\II-\SSth}&\leq \norminf{\II-\BB}+\norminf{\WWo}\norminf{\SSoi}\norminf{\SSt}\\
&\qquad+\norminf{\WWt}\norminf{\DtWti}(\norminf{\DtWo}\norminf{\SSoi}\norminf{\SSt}+\norminf{\DtB}),
\end{split}
\end{equation}
and
\begin{equation}\label{eq:matrixformulalast}
\norminf{\DtWti}\leq(1-\norminf{\II-\DtWt})^{-1}.
\end{equation}
From the last result we can deduce that $\abs{\bcoeff_i}\geq 1-\norminf{\SSthi}\norminf{I-\SSth}=\bcoeffmin$.
\end{restatable}

\begin{proof}
If $\DtWti$ and $\SSoi$ exist, and since $\SSo$ is the Schur complement of $\DtWt$, the block matrix inversion formula gives
\begin{equation}
\WWpi=
\begin{bmatrix}
\DoWo&\DoWt\\
\DtWo&\DtWt
\end{bmatrix}^{-1}=
\begin{bmatrix}
\SSoi&-\SSoi\DoWt\DtWti\\
-\DtWti\DtWo\SSoi&\DtWti(I+\DtWo\SSoi\DoWt\DtWti)
\end{bmatrix}.
\end{equation}

Since $\SSth$ is assumed invertible and is the Schur complement of $\WWp$ as mentioned in \Cref{lem:SchurInv}, the block matrix inversion formula gives $\MM^{-1}$:
\begin{equation}
\MM^{-1}=
\begin{bmatrix}
\SSthi&-\SSthi\WW\WWpi\\
-\WWpi\BBp\SSthi&\WWpi+\WWpi\BBp\SSthi\WW\WWpi
\end{bmatrix}.
\end{equation}
Then,
\begin{equation}
\begin{bmatrix}
\bcoeff\\
\Gamma
\end{bmatrix}
=
\MM^{-1}
\begin{bmatrix}
\signs\\
0
\end{bmatrix}
=
\begin{bmatrix}
\SSthi\signs\\
-\WWpi\BBp\SSthi\signs
\end{bmatrix}.
\end{equation}
Since
\begin{align}
\begin{split}
\WWpi\BBp&=
\begin{bmatrix}
\SSoi(\DoB-\DoWt\DtWti\DtB)\\
\DtWti\DtB-\DtWti\DtWo\SSoi\DoB+\DtWti\DtWo\SSoi\DoWt\DtWti\DtB
\end{bmatrix}\\
&=\begin{bmatrix}
\SSoi\SSt\\
\DtWti\DtB-\DtWti\DtWo\SSoi\SSt
\end{bmatrix},
\end{split}
\end{align}
and $\norminf{\signs}=1$, we have
\begin{equation}\label{eq:alphabound}
\norminf{\bcoeff}\leq\norminf{\SSthi}
\end{equation}
and
\begin{equation}\label{eq:betabound}
\norminf{\wocoeff}, \norminf{\wtcoeff}\leq\norminf{\SSoi}\norminf{\SSt}\norminf{\SSthi}.
\end{equation}
Additionally
\begin{equation}
\bcoeff-\signs=(\SSthi-I)\signs=\SSthi(I-\SSth)\signs
\end{equation}
so
\begin{equation}\label{eq:alpha1bound}
\abs{\alpha_i-\signs_i}\leq\norminf{\SSthi}\norminf{I-\SSth}.
\end{equation}
\eqref{eq:matrixformulafirst}--\eqref{eq:matrixformulalast} are easily derived by the matrix definitions, \eqref{eq:neumann}, and the triangle inequality and submultiplicativity of the $\infty$-norm.
\end{proof}

\subsection{Proof of \Cref{lem:BumpTail}}\label{sec:BumpTailProof}

First we derive some simple inequalities and introduce lemmas that will aid us. If the distance of the sample separation (grid spacing) is denoted by $\sampleprox$, then ${\abs{s_{j,(k)}^i-t_{j,(k)}}\leq\sampleprox}$ for spike $t_j=(t_{j,(1)},t_{j,(2)})$, sample point $i=1,2,3$ and coordinate $k=1,2$, so the distance between coordinates of the sample and spike is at most $\gridsep$ as well. Without loss of generality and for ease of notation we can assume $t_j$ sits at the origin. We first list formulas for the partial derivatives of the Gaussian kernel at point $t=(t_{(1)},t_{(2)})$ that will be useful in what follows:
\begin{gather}
\partial_xK(t)=-t_{(1)}\exp\left(-\frac{\|t\|^2}{2}\right)=-t_{(1)}K(t)\\
\partial_yK(t)=-t_{(2)}\exp\left(-\frac{\|t\|^2}{2}\right)=-t_{(2)}K(t)\\
\partial_{xx}K(t)=(t_{(1)}^2-1)\exp\left(-\frac{\|t\|^2}{2}\right)=(t_{(1)}^2-1)K(t)\\
\partial_{yy}K(t)=(t_{(2)}^2-1)\exp\left(-\frac{\|t\|^2}{2}\right)=(t_{(2)}^2-1)K(t)\\
\partial_{xy}K(t)=t_{(1)}t_{(2)}\exp\left(-\frac{\|t\|^2}{2}\right)=t_{(1)}t_{(2)}K(t).
\end{gather}
By applying the triangle inequality, and using the fact that $t_j$ is at the origin, we obtain the following simple bounds for points $t$
satisfying $\norm{t}\geq10$ and $\sampleprox\leq1$:
\begin{gather}
|t_{(1)}-s^i_{(1)}|\leq\|t\|+\sampleprox\leq2\|t\|,\label{eq:aux1}\\
|t_{(2)}-s^i_{(2)}|\leq\|t\|+\sampleprox\leq2\|t\|,\label{eq:aux2}\\
|t_{(1)}-s^i_{(1)}|^2-1\leq(\|t\|+\sampleprox)^2-1\leq2\|t\|^2,\label{eq:aux3}\\
|t_{(2)}-s^i_{(2)}|^2-1\leq(\|t\|+\sampleprox)^2-1\leq2\|t\|^2,\label{eq:aux4}\\
|t_{(1)}-s^i_{(1)}||t_{(2)}-s^i_{(2)}|\leq(\|t\|+\sampleprox)^2\leq2\|t\|^2,\label{eq:aux5}
\end{gather}
for $i=1,2,3$, where we have dropped the subscript $j$ on the sample points for clarity.
\begin{restatable}[Proof in \Cref{sec:BumpBoundsProof}]{lemma}{BumpBounds}\label{lem:BumpBounds} 
Let $B(t)$ denote the bump function corresponding to a spike at the origin and some configuration of the three closest samples.
If $\norm{t}\geq10$ and $\sampleprox\leq 1$,
then the absolute values of $B(t)$ and its first and second partial derivatives are all bounded by
\begin{equation}\label{eq:bumpbound}
g(t)=6\|t\|^2\exp\left(-\frac{\|t\|^2}{2}+\sqrt{2}\sampleprox\|t\|\right),
\end{equation}
for any configuration of the three closest samples. For $\norm{t}\geq 10$ and $\gridsep\leq 1$, this is less than $10^{-12}$.
\end{restatable}
\begin{proof}
By \eqref{eq:coeffmat}
$$\bcone=\frac{D_1}{D}\exp\left(\frac{\|s^1\|^2}{2}\right)$$
so using \Cref{lem:crossprod} and since $\norm{s^i}\leq\sqrt{2}\gridsep$ for all three samples,
\begin{align}
|\bcone K(t-s^1)|&=\left|\frac{D_1}{D}\right|\exp\left(\|s^1\|^2/2\right)\exp\left(-\|t-s^1\|^2/2\right)\\
&\leq \exp\left(\frac{\|s^1\|^2-\|s^1\|^2-\|t\|^2+2\|s^1\|\|t\|}{2}\right)\\
&\leq\exp\left(-\frac{\|t\|^2}{2}+\sqrt{2}\sampleprox\|t\|\right).\label{eq:coeffkernelbd}
\end{align}
Also by \eqref{eq:aux1}
\begin{align}
|\bcone \partial_xK(t-s^1)|&=|(t_{(1)}-s^1_{(1)})\bcone K(t-s^1)|\\
&\leq2\|t\|\exp\left(-\frac{\|t\|^2}{2}+\sqrt{2}\sampleprox\|t\|\right).
\end{align}
The same holds for $\abs{\bcone \partial_yK(t-s^1)}$. By \eqref{eq:aux3},
\begin{align}
\abs{\bcone \partial_{xx}K(t-s^1)}&=\abs{((t_{(1)}-s^1_{(1)})^2-1)\bcone K(t-s^1)}\\
&\leq2\norm{t}^2\exp\left(-\frac{\norm{t}^2}{2}+\sqrt{2}\sampleprox\norm{t}\right).
\end{align}
The same holds for $\abs{\bcone\partial_{yy}K(t-s^1)}$ and $\abs{\bcone\partial_{xy}K(t-s^1)}$. The same bounds also hold for the second and third samples, with $\bcone$ replaced with $\bctwo$ and $\bcthree$, respectively.
Thus for $\norm{t}\geq10$ and any configuration of the three nearest samples,
\begin{align}
\abs{B(t)}&\leq\abs{\bcone K(t-s^1)}+\abs{\bctwo K(t-s^2)}+\abs{\bcthree K(t-s^3)}\\
&\leq3\exp\left(-\frac{\norm{t}^2}{2}+\sqrt{2}\sampleprox\norm{t}\right)\\
&\leq6\|t\|^2\exp\left(-\frac{\|t\|^2}{2}+\sqrt{2}\sampleprox\|t\|\right),\\
\abs{\partial_{x}B(t)}&\leq|\bcone\partial_{x}K(t-s^1)|+|\bctwo\partial_{x}K(t-s^2)|+|\bcthree\partial_{x}K(t-s^3)|\\
&\leq6\|t\|\exp\left(-\frac{\|t\|^2}{2}+\sqrt{2}\sampleprox\|t\|\right)\\
&\leq6\|t\|^2\exp\left(-\frac{\|t\|^2}{2}+\sqrt{2}\sampleprox\|t\|\right),\\
\abs{\partial_{xx}B(t)}&\leq|\bcone\partial_{xx}K(t-s^1)|+|\bctwo\partial_{xx}K(t-s^2)|+|\bcthree\partial_{xx}K(t-s^3)|\\
&\leq6\|t\|^2\exp\left(-\frac{\|t\|^2}{2}+\sqrt{2}\sampleprox\|t\|\right).
\end{align}
Similar reasoning shows the same bound holds for $\abs{\partial_yB(t)}$, $\abs{\partial_{yy}B(t)}$ and $\abs{\partial_{xy}B(t)}$.
\end{proof}

\begin{restatable}{lemma}{WaveBounds}\label{lem:WaveBounds} 
Let $W^i(t)$, for $i=1,2$, denote the $i$th wave function corresponding to a spike at the origin and some configuration of the three closest samples.
If $\norm{t}\geq10$ and $\sampleprox\leq 1$,
then the absolute values of $W^i(t)$ and its first and second partial derivatives are all bounded by \begin{equation}\label{eq:wavebound}
g(t)=\frac{6\|t\|^2}{\gridsep}\exp\left(-\frac{\|t\|^2}{2}+\sqrt{2}\sampleprox\|t\|\right),
\end{equation}
for any configuration of the three closest samples.
For $10^{-2}\leq\gridsep\leq 1$, this is less than $2\cdot 10^{-9}$.
\end{restatable}
\begin{proof}
By \eqref{eq:coeffmat}
\begin{equation}
\wocone=\frac{s^2_{(2)}-s^3_{(2)}}{D}\exp\left(\frac{\|s^1\|^2}{2}\right).
\end{equation}
Since $|s^2_{(2)}-s^3_{(2)}|\leq\sampleprox$ and $|D|=\sampleprox^2$ by \Cref{lem:crossprod},
\begin{align}
|\wocone K(t-s^1)|&=\left|\frac{s^2_{(2)}-s^3_{(2)}}{D}\right|\exp\left(\|s^1\|^2/2\right)\exp\left(-\|t-s^1\|^2/2\right)\\
&\leq\frac{1}{\sampleprox}\exp\left(-\frac{\|t\|^2}{2}+\sqrt{2}\sampleprox\|t\|\right).\label{eq:wcoeffkernelbd}
\end{align}
Also by \eqref{eq:aux1}
\begin{align}
\abs{\wocone\partial_x K(t-s^1)}&=\abs{(t_{(1)}-s^1_{(1)})\wocone K(t-s^1)}\\
&\leq\frac{2\|t\|}{\sampleprox}\exp\left(-\frac{\|t\|^2}{2}+\sqrt{2}\sampleprox\|t\|\right),
\end{align}
and the same bound holds for $\abs{\wocone\partial_y K(t-s^1)}$. By \eqref{eq:aux3}
\begin{align}
\abs{\wocone\partial_{xx}K(t-s^1)}&=\abs{((t_{(1)}-s^1_{(1)})^2-1)\wocone K(t-s^1)}\\
&\leq\frac{2\|t\|^2}{\sampleprox}\exp\left(-\frac{\|t\|^2}{2}+\sqrt{2}\sampleprox\|t\|\right).
\end{align}
The same holds for $\abs{\wocone\partial_{yy}K(t-s^1)}$ and $\abs{\wocone\partial_{xy}K(t-s^1)}$. The same bounds also hold for the second and third samples, with $\bcone$ replaced with $\bctwo$ and $\bcthree$, respectively. Thus for $\norm{t}\geq10$, $i=1,2$, and any configuration of the three closest samples,
\begin{align}
|W^i(t)|&\leq|\wocone K(t-s^1)|+|\woctwo K(t-s^2)|+|\wocthree K(t-s^3)|\\
&\leq\frac{3}{\sampleprox}\exp\left(-\frac{\|t\|^2}{2}+\sqrt{2}\sampleprox\|t\|\right)\\
&\leq\frac{6\|t\|^2}{\sampleprox}\exp\left(-\frac{\|t\|^2}{2}+\sqrt{2}\sampleprox\|t\|\right),\\
|\partial_{x}W^i(t)|&\leq|\wocone\partial_{x}K(t-s^1)|+|\woctwo\partial_{x}K(t-s^2)|+|\wocthree\partial_{x}K(t-s^3)|\\
&\leq\frac{6\|t\|}{\sampleprox}\exp\left(-\frac{\|t\|^2}{2}+\sqrt{2}\sampleprox\|t\|\right)\\
&\leq\frac{6\|t\|^2}{\sampleprox}\exp\left(-\frac{\|t\|^2}{2}+\sqrt{2}\sampleprox\|t\|\right),\\
|\partial_{xx}W^i(t)|&\leq|\wocone\partial_{xx}K(t-s^1)|+|\woctwo\partial_{xx}K(t-s^2)|+|\wocthree\partial_{xx}K(t-s^3)|\\
&\leq\frac{6\|t\|^2}{\sampleprox}\exp\left(-\frac{\|t\|^2}{2}+\sqrt{2}\sampleprox\|t\|\right).
\end{align}
Similar reasoning shows the same bound holds for $\abs{\partial_y W^i(t)}$, $\abs{\partial_{yy}W^i(t)}$ and $\abs{\partial_{xy}W^i(t)}$.
\end{proof}

\BumpTail*
\begin{figure}[t]
\centering
\includegraphics{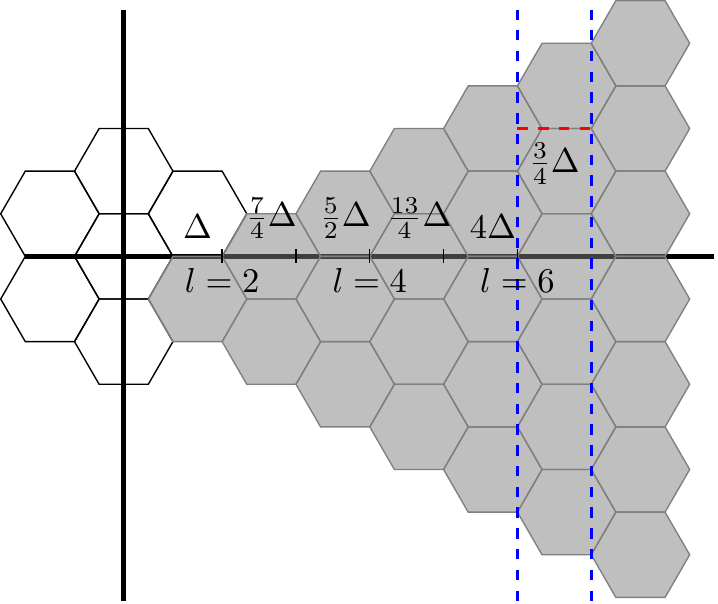}
\caption[Bounds on Distant Spikes]{Hexagonal grid rotated to show distance bounds for spike layers. We note that each layer of cells sits $3\spikesep/4$ further from the origin than the previous layer. }\label{fig:cellsep}
\end{figure}

\begin{proof}
If $h(x)=x^2\exp(-x^2/2+\sqrt{2}\sampleprox x)$, and $x>\sqrt{2}\sampleprox+2$,
\begin{align}
h^\prime(x)&=2x\exp\left(-\frac{x^2}{2}+\sqrt{2}\sampleprox x\right)+x^2(-x+\sqrt{2}\sampleprox)\exp\left(-\frac{x^2}{2}+\sqrt{2}\sampleprox x\right)\\
&=(-x^3+\sqrt{2}\sampleprox x^2+2x)\exp\left(-\frac{x^2}{2}+\sqrt{2}\sampleprox x\right)\\
&\leq(-x^3+(\sqrt{2}\sampleprox+2)x^2)\exp\left(-\frac{x^2}{2}+\sqrt{2}\sampleprox x\right)\\
&<0.
\end{align}
Thus for $\norm{t}\geq 10$ and $\sampleprox\leq 1$, the bound \eqref{eq:bumpbound} is strictly decreasing \bdb{with $\norm{t}\rightarrow\infty$}.
By rotating the plane we note that each layer of $U_i$'s is positioned $3\spikesep/4$ further from the origin than the previous layer as \Cref{fig:cellsep} indicates. Spikes $t_j$ in cells in layer $l=2$ have $\norm{t_{j}}\geq\spikesep$ so $\norm{t_{i}}\geq (3l-2)\spikesep/4$ for spikes $t_i$ in layer $l$.
Thus if $\spikesep\geq2$, spikes in the $l$\textsuperscript{th} layer ($l\geq2$) satisfy
$\norm{t_{j}}\ \geq\ (3l-2)\spikesep/4\ \geq\ 3l/2-1$.
The distance between a spike $t_j$ in layer $l$ and a point $z$ with $\norm{z}\leq\spikesep$ is given by
\begin{equation}\label{eq:LayerIneq}
\norm{t_{j}-z}\geq(3l-2)\spikesep/4-\spikesep\geq 3l/2-3.
\end{equation}
For layers $l\geq 9$, $\norm{t_{j}-z}\geq10$.  
Below, to obtain an upper bound, we will assume that each layer $l$ has $6l$ spikes, the maximum possible number (visible in \Cref{fig:cellsep}).  Let $\UU_l$ denote the union of all hexagonal cells in the $l$th layer.
\begin{align}
\sum_{t_j\in \UU_l}\abs{f(t_{j}-z)}&\leq\sum_{t_j\in \UU_l}6\|t_{j}-z\|^2\exp\left(-\frac{\|t_{j}-z\|^2}{2}+\sqrt{2}\sampleprox\norm{t_{j}-z}\right)\\
&\leq 6l\cdot 6(3l/2-3)^2\exp\left(-\frac{(3l/2-3)^2}{2}+\sqrt{2}\sampleprox(3l/2-3)\right)\\
&\leq 36(3l/2-3)^3\exp\left(-\frac{(3l/2-3)^2}{2}+\sqrt{2}\sampleprox(3l/2-3)\right)\label{eq:bumpbound2}
\end{align}
where the first inequality follows from \Cref{lem:BumpBounds}, \bdb{the second from the fact that $h$ is decreasing and \eqref{eq:LayerIneq}, and the third from $l\geq 9$}. Let $p(a):=36a^3\exp(-a^2/2+\sqrt{2}\sampleprox a)$. If $a\geq 10$ and $\sampleprox\leq1$
\begin{align}
\frac{p(a+1)}{p(a)}&=\frac{36(a+1)^3\exp(-(a+1)^2/2+\sqrt{2}\sampleprox(a+1))}{36a^3\exp(-a^2/2+\sqrt{2}\sampleprox a)}\\
&=\frac{(a+1)^3}{a^3}\exp(-((a+1)^2-a^2)/2+\sqrt{2}\sampleprox(a+1)-\sqrt{2}\sampleprox a)\\
&=\left(1+\frac{3}{a}+\frac{3}{a^2}+\frac{1}{a^3}\right)\exp(-a-1/2+\sqrt{2}\sampleprox)\\
&\leq\frac{2}{e^{a-\sqrt{2}\sampleprox+1/2}}\leq\frac{2}{e^{9}} < \frac{1}{2}.
\end{align}
With this
\begin{align}
\sum_{l=9}^\infty p(3l/2-3)&\leq\sum_{m=10}^\infty p(m)= p(10)\sum_{m=10}^\infty \frac{p(m)}{p(10)}\\
&< p(10)\sum_{m=0}^\infty 2^{-i}=2p(10)\\
&<2\times10^{-11}.
\end{align}
Thus the sum of bumps or derivatives at spikes in layers $9$ and above can be bounded by
\begin{align}
\sum_{l=9}^\infty\sum_{t_j\in\UU_l}\abs{f(t_{j})}&\leq\sum_{l=9}^\infty p(3l/2-3) &&\text{by \eqref{eq:bumpbound2}} \\
&<\ 2\times10^{-11}.
\end{align}
This also holds using $\abs{f(t_{l,j}-z)}$ for points $\norm{z}\leq\spikesep$.

The same argument from \Cref{lem:BumpTail} together with \Cref{lem:WaveBounds} can be used to show that the sum of the waves or derivatives of such at spikes in distant $U_i$ ($l\geq9$) is bounded by
\begin{equation}
\sum_{l=9}^\infty\sum_{j=1}^{6l}|f(t_{l,j})|\ \leq\ \frac{1}{\sampleprox}\sum_{l=9}^\infty p(3l/2-3)\ <\ \frac{2}{\sampleprox}\times10^{-11}.
\end{equation}
If $\gridsep>10^{-2}$, this is less than $2\times10^{-9}$.
\end{proof}

\begin{lemma}\label{lem:EVTail}
If $\norm{t}\geq10$ and $\sampleprox\leq1$,
then $\abs{v^T\nabla^2 B_1(t) v}$ is bounded by 
\begin{equation}\label{eq:BumpEVbound}
f(t)=6\|t\|^2\exp\left(-\frac{\|t\|^2}{2}+\sqrt{2}\sampleprox\|t\|\right)
\end{equation}
and $\abs{v^T\nabla^2 W^1_1(t) v}$ and $\abs{v^T\nabla^2 W^2_j(t)v}$ by 
\begin{equation}
f(t)=\frac{6\|t\|^2}{\gridsep}\exp\left(-\frac{\|t\|^2}{2}+\sqrt{2}\sampleprox\|t\|\right).
\end{equation}
Consequently, if $\UU_{\geq 9}$ denotes the union of all hexagonal cells in layers nine and higher and $z$ is any point with $\norm{z}_2\leq\spikesep(T)$, then
\begin{equation}
\sum_{t_k\in T\cap\,\UU_{\geq9}}\abs{f(t_k-z)}< 2\times10^{-11}=:\epsbump\mbox{\quad and\quad}
\sum_{t_k\in T\cap\,\UU_{\geq9}}\abs{g(t_k-z)}< 2\times 10^{-9}=:\epswave,
\end{equation}
where $f(t)=\abs{v^T\nabla^2 B_1(t) v}$
and $g(t)=\abs{v^T\nabla^2W^i_j(t) v}$ for $i=1,2$ and $\norm{v}=1$.
\end{lemma}
\begin{proof}
Consider the first term in \eqref{eq:BumpEVAbsSum}, $|\bcone_j|\max(\|s^1_j-t\|^2-1,1)e^{-\norm{s^1_j-t}^2/2}$.
From \eqref{eq:coeffkernelbd} we get that
\begin{equation}
\abs{\bcone_j} e^{-\norm{s^1_j-t}^2/2}\leq\exp\left(-\frac{\|t\|^2}{2}+\sqrt{2}\sampleprox\|t\|\right),
\end{equation}
and \eqref{eq:aux5} gives
\begin{equation}
\max(\abs{\norm{s^1_j-t}^2-1},1)\leq 2\norm{t}^2.
\end{equation}
The same holds for the second and third summands and so \eqref{eq:BumpEVbound} holds.
We can separate the three terms of $W^1_j$ and $W^2_j$ in the same way to get a bound for $\enva{\lambda(W^i)}(r)$ similarly using \eqref{eq:wcoeffkernelbd} and \eqref{eq:aux5}.

From there, \Cref{lem:BumpTail} can be extended so that $f$ can be the largest absolute eigenvalues of either the bump or wave functions respectively. Thus the contributions to the value of $v^T\nabla^2Q(t) v$ (where $\norm{v}=1$) from spikes beyond layer $l=9$ are less than $\norminf{\bcoeff}\epsbump$ for all bumps and $(\norminf{\wocoeff}+\norminf{\wtcoeff})\epswave$ for all waves.
\end{proof}

\section{Proof of \Cref{lem:Qtri}}\label{sec:QtriProof}

\subsection{\Cref{eq:QBound} and \Cref{eq:QLowerBound}}
Since the envelopes in \eqref{eq:Envelopes} are monotonically decreasing, for $t\in S:=[a,b]$,
\begin{align}
\abs{Q(t)}&=\abs{\sum_{t_j\in T}\bcoeff_j B_j(t)+\wocoeff_j W^1_j(t)+\wtcoeff_j W^2_j(t)}\\
&\leq\sum_{t_j\in T}\norminf{\bcoeff}\enva{B}(\norm{t-t_j})+\norminf{\wocoeff}\enva{W^1}(\norm{t-t_j})+\norminf{\wtcoeff}\enva{W^2}(\norm{t-t_j})\\
\begin{split}\label{eq:Qenvbound}
&\leq\norminf{\bcoeff}\enva{B}(d(t_1,S))+\norminf{\wocoeff}\enva{W^1}(d(t_1,S))+\norminf{\wtcoeff}\enva{W^2}(d(t_1,S))\\
&\qquad+\sum_{U\in \CC_{\leq 8}}\norminf{\bcoeff}\enva{B}(d(S,U))+\norminf{\wocoeff}\enva{W^1}(d(S,U))+\norminf{\wtcoeff}\enva{W^2}(d(S,U))\\
&\qquad+\norminf{\bcoeff}\epsbump+\norminf{\wocoeff}\epswave+\norminf{\wtcoeff}\epswave.
\end{split}
\end{align}
By \Cref{lem:BumpTail} and that $\norminf{\bcoeff}\leq 2$, $\norminf{\wocoeff}$ and $\norminf{\wtcoeff}\leq 1$ (as plotted in \Cref{fig:NormsCoeffs}) we have
\begin{equation}
\norminf{\bcoeff}\epsbump+\norminf{\wocoeff}\epswave+\norminf{\wtcoeff}\epswave < 10^{-9}.
\end{equation}
Combining this with $d(t_1,S)=a$ and $d_U=d(S,U)$ gives \eqref{eq:QBound}.

Recalling that $\bcoeff_1 B_1(t)\geq 0$ from our assumption that $Q(t_1)=1$, we get the lower bound \eqref{eq:QLowerBound} easily from \eqref{eq:QBound} by removing the first term.
This lower bound, denoted as $Q_{\LB}$, is used to show that $Q(t)>-1$ up to $u_2$.

\subsection{\Cref{eq:RawGradBound}}\label{sec:DQ}

Since $t_1$ is the origin, $\nabla Q\cdot t/\norm{t}$ is the radially outward directional derivative along the direction of $t$.
For convenience denote $\hat{t}=t/\norm{t}$.
Note
\begin{equation}
\abs{\nabla B_j(t)\cdot\hat{t}}\leq\norm{\nabla B_j(t)}\leq(\enva{\partial_x B}(\norm{t})^2+\enva{\partial_y B}(\norm{t})^2)^{1/2}=:\norm{\nabla B}_\downarrow(\norm{t}),
\end{equation}
which is monotone decreasing since both envelopes $\enva{\partial_x B}$ and $\enva{\partial_y B}$ are, and analogously for the directional derivatives of the two waves.
Then
\begin{align}
\nabla Q(t)\cdot \hat{t}&=\sum_{t_j\in T}\nabla [\bcoeff_j B_j(t)+\wocoeff_j W^1_j(t)+\wtcoeff_j W^2_j(t)]\cdot\hat{t}\\
\begin{split}
&\leq\bcoeff_1\nabla B_1(t)\cdot\hat{t}+\norminf{\wocoeff}\norm{\nabla W^1}_\downarrow(\norm{t-t_1})+\norminf{\wtcoeff}\norm{\nabla W^2}_\downarrow(\norm{t-t_1})\\
&\quad+\sum_{t_j\in T\setminus\{t_1\}}\norminf{\bcoeff}\norm{\nabla B}_\downarrow(\norm{t-t_j})+\norminf{\wocoeff}\norm{\nabla W^1}_\downarrow(\norm{t-t_j})+\norminf{\wtcoeff}\norm{\nabla W^2}_\downarrow(\norm{t-t_j}).
\end{split}\label{eq:Dbound1}
\end{align}
We use \Cref{lem:BumpTail} to bound contributions from spikes outside $U\in \CC_{\leq 8}$. Note
\begin{equation}
(\enva{\partial_x B}(s)^2+\enva{\partial_y B}(s)^2)^{1/2}\leq
\enva{\partial_x B}(s)+\enva{\partial_y B}(s)
\end{equation}
by squaring both sides.
Thus
\begin{align}
\begin{split}
\sum_{\substack{t_j\in U\\U\in\CC_{\geq 9}}}&\norminf{\bcoeff}\norm{\nabla B}_\downarrow(\norm{t-t_j})+\norminf{\wocoeff}\norm{\nabla W^1}_\downarrow(\norm{t-t_j})+\norminf{\wtcoeff}\norm{\nabla W^2}_\downarrow(\norm{t-t_j})\\
&\leq\sum_{\substack{t_j\in U\\U\in\CC_{\geq 9}}}
\norminf{\bcoeff}\enva{\partial_x B}(\norm{t-t_j})+\norminf{\bcoeff}\enva{\partial_y B}(\norm{t-t_j})\\
&\qquad+\norminf{\wocoeff}\enva{\partial_x W^1}(\norm{t-t_j})+\norminf{\wocoeff}\enva{\partial_y W^1}(\norm{t-t_j})\\
&\qquad+\norminf{\wtcoeff}\enva{\partial_x W^2}(\norm{t-t_j})+\norminf{\wtcoeff}\enva{\partial_y W^2}(\norm{t-t_j})
\end{split}\\
&\leq 2\norminf{\bcoeff}\epsbump+2\norminf{\wocoeff}\epswave+2\norminf{\wtcoeff}\epswave< 10^{-9}.
\end{align}
As with \eqref{eq:QBound}, since $a\leq\norm{t-t_1}$ for all $t\in[a,b]$ and $d_U\leq\norm{t-t_j}$ for $t_j\in U$, from \eqref{eq:Dbound1} we get
\begin{equation}
\begin{aligned}
\nabla Q(t)\cdot \hat{t}
&\leq\bcoeff_1\nabla B_1(t)\cdot\hat{t}+\norminf{\wocoeff}\norm{\nabla W^1}_\downarrow(a)+\norminf{\wtcoeff}\norm{\nabla W^2}_\downarrow(a)\\
&\quad+\sum_{U\in\CC_{\leq 8}}\norminf{\bcoeff}\norm{\nabla B}_\downarrow(d_U)+\norminf{\wocoeff}\norm{\nabla W^1}_\downarrow(d_U)+\norminf{\wtcoeff}\norm{\nabla W^2}_\downarrow(d_U)
+\epsilon.
\end{aligned}\label{eq:Dbound2}
\end{equation}
Lastly, we can bound $\bcoeff_1\nabla B_1(t)\cdot\hat{t}$ using \eqref{eq:RawEnvelopes}.
Note $\bcoeff_1\geq 0$ since $Q(t_1)=1$ and recall $\bcoeffmin$ represents the smallest magnitude that $\bcoeff_1$ can be. 
Thus when $\nabla B_1(t)\cdot\hat{t}\leq 0$,
\begin{equation}
\bcoeff_1\nabla B_1(t)\cdot\hat{t}\leq \bcoeffmin D(B)_\infty(\norm{t}),
\end{equation}
and when $\nabla B_1(t)\cdot\hat{t}\geq 0$
\begin{equation}
\bcoeff_1\nabla B_1(t)\cdot\hat{t}\leq \norminf{\bcoeff} D(B)_\infty(\norm{t}).
\end{equation}
Denote $\omega:=\sup_{t\in[a,b]}\env{D (B)}(\|t\|)$ so that for all $t\in[a,b]$
\begin{equation}
\bcoeff_1\nabla B_1(t)\cdot\hat{t}\leq \max(\bcoeff_{\LB}\ \omega,\norminf{\bcoeff}\omega).
\end{equation}
Substituting this into \eqref{eq:Dbound2} yields \eqref{eq:RawGradBound}.

\subsection{\Cref{eq:RawEigBound}}\label{sec:DDQ}

Let $\nabla^2 B_j$ and $\nabla^2 W_j^k$ for $k\in\{1,2\}$ denote the Hessians of the bump and two wave functions for each spike.
By decomposing $Q$'s Hessian into its bump and wave components and using envelopes in \eqref{eq:Envelopes} and \eqref{eq:RawEnvelopes}, for any unit vector $v$
\begin{align}
v^T\nabla^2Q(t)v&=\ \sum_{t_j\in T} v^T(\bcoeff_j \nabla^2 B_j(t)+\wocoeff_j\nabla^2 W^1_j(t)+\wtcoeff_j\nabla^2 W^2_j(t))\ v\\
\begin{split}
&\leq \bcoeff_1 v^T \nabla^2 B_1(t) v+\norminf{\wocoeff}\ltmaxwo(\norm{t-t_1})+\norminf{\wtcoeff}\ltmaxwt(\norm{t-t_1})\\
&\qquad+\sum_{t_j\in T\setminus\{t_1\}} \norminf{\bcoeff}\ltamaxb(\norm{t-t_j})+\norminf{\wocoeff}\ltmaxwo(\norm{t-t_j})\\
&\qquad\qquad+\norminf{\wtcoeff}\ltmaxwt(\norm{t-t_j}).\label{eq:Hbound}
\end{split}\\
\begin{split}
&\leq \bcoeff_1 v^T \nabla^2 B_1(t) v+\norminf{\wocoeff}\ltmaxwo(a)+\norminf{\wtcoeff}\ltmaxwt(a)\\
&\qquad+\sum_{U\in \CC_{\leq 8}}\norminf{\bcoeff}\ltamaxb(d_U)+\norminf{\wocoeff}\ltmaxwo(d_U)+\norminf{\wtcoeff}\ltmaxwt(d_U)+\epsilon\label{eq:Hbound1}
\end{split}
\end{align}
We get \eqref{eq:Hbound1} by noting the envelope functions decrease monotonically and $a\leq\norm{t-t_1}$ and $d_U\leq\norm{t-t_j}$ for $t_j\in U$, and since the combined contributions for spikes in $U$ outside $\CC_{\leq 8}$ are less than $10^{-9}$ by \Cref{lem:EVTail}.

The term $v^T \nabla^2 B_1(t)v$ from the bump at the origin can be bounded using \eqref{eq:RawEnvelopes}.
As with \eqref{eq:RawGradBound} $\bcoeff_1\geq 0$ since $Q(t_1)=1$ and $\bcoeff_{\LB}\leq\abs{\bcoeff_1}$. 
Thus when $v^T \nabla^2 B_1(t) v\leq 0$
\begin{equation}
\bcoeff_1 v^T \nabla^2 B_1(t) v\leq\bcoeffmin\env{\lambda(B)}(\norm{t}).
\end{equation}
and when $v^T \nabla^2 B_1(t) v\geq 0$,
\begin{equation}
\bcoeff_1v^T \nabla^2 B_1(t) v\leq \norminf{\bcoeff}\env{\lambda(B)(\norm{t})}.
\end{equation}
By defining $\eta:=\sup_{t\in[a,b]}\env{\lambda(B)}(\|t\|)$, then for all $t\in[a,b]$
\begin{equation}
\bcoeff_1 v^T \nabla^2 B_1(t)v\leq \max(\bcoeffmin\ \eta,\norminf{\bcoeff}\eta).
\end{equation}
Substituting this in \eqref{eq:Hbound1} yields \eqref{eq:RawEigBound}:
\begin{equation}
\begin{aligned}
v^T\nabla^2Q(t)v&\leq \max(\bcoeffmin\ \eta,\norminf{\bcoeff}\eta)
+\norminf{\wocoeff}\envEV{W^1}(a)+\norminf{\wtcoeff}\envEV{W^2}(a)\\
&\qquad+\sum_{U\in \CC_{\leq 8}}\norminf{\bcoeff}\envEV{B}(d_U)+\norminf{\wocoeff}\envEV{W^1}(d_U)+\norminf{\wtcoeff}\envEV{W^2}(d_U)+\epsilon.
\end{aligned}
\end{equation}
If instead $Q(t_1)=-1$, a similar lower bound holds as we show $Q$ to be positive definite and swap signs, inequalities and maximums for minimums accordingly.

\subsection{A Note on Rotational Invariance}

As we point out before introducing \Cref{lem:Qtri}, we simplify our argument by assuming that $[a,b]$ is an interval along the positive horizontal axis within the disk of radius $\spikesep(T)$.
Since the envelope functions are radially symmetric, the bounds \eqref{eq:QBound}, \eqref{eq:RawGradBound} and \eqref{eq:RawEigBound} are not limited to points on the positive horizontal axis but generalize to all points within the disk.
We show this in detail for \eqref{eq:QBound}.
If $\phi$ represents a rotation of the plane then by applying $\phi$ to our original partition $\{\phi(U)\}$ is another partition.
Let $d_{\phi(U)}=d(\phi(U),\phi([a,b]))$ and $t\in[a,b]$ on the positive horizontal axis as before.
Then,
\begin{align}
\abs{Q(\phi(t))}&=\abs{\sum_{j=0}\bcoeff_j B_j(\phi(t))+\wocoeff_j W^1_j(\phi(t))+\wtcoeff_j W^2_j(\phi(t))}\\
\begin{split}
&\leq\norminf{\bcoeff}\enva{B}(a)+\norminf{\wocoeff}\enva{W^1}(a)+\norminf{\wtcoeff}\enva{W^2}(a)\\
&\qquad+\sum_{U\in\CC_{\leq8}}\norminf{\bcoeff}\enva{B}(d_{\phi (U)})+\norminf{\wocoeff}\enva{W^1}(d({\phi(U),})+\norminf{\wtcoeff}\enva{W^2}(d_{\phi(U)})+\epsilon
\end{split}
\\
\begin{split}
&\leq\norminf{\bcoeff}\enva{B}(a)+\norminf{\wocoeff}\enva{W^1}(a)+\norminf{\wtcoeff}\enva{W^2}(a)\\
&\qquad+\sum_{U\in\CC_{\leq8}}\norminf{\bcoeff}\enva{B}(d_U)+\norminf{\wocoeff}\enva{W^1}(d_U)+\norminf{\wtcoeff}\enva{W^2}(d_U)+\epsilon
\end{split}
\end{align}
which is the bound in \eqref{eq:QBound}. A similar rotational invariance argument can be applied to \eqref{eq:RawGradBound} and \eqref{eq:RawEigBound} to show that they hold for all points in the disk of radius $\spikesep(T)$.

\end{document}